\documentclass[12pt]{amsart}
\allowdisplaybreaks
\usepackage{mathrsfs}
\usepackage{cases}
\usepackage{epic}
\usepackage{amsfonts}
\usepackage{graphicx}
\usepackage{amsmath}
\usepackage{amssymb, upgreek}
\usepackage{bm}
\usepackage{latexsym,todonotes}
\usepackage{pdflscape}
\usepackage[all]{xypic}
\usepackage[all]{xy}
\usepackage{color}
\usepackage{colordvi}
\usepackage{multicol}
\usepackage[normalem]{ulem}
\usepackage[linktocpage=true]{hyperref}
\hypersetup{colorlinks,linkcolor=blue,urlcolor=cyan,citecolor=blue}

\newcommand{\nc}{\newcommand}
\nc{\browntext}[1]{\textcolor{brown}{#1}}
\nc{\greentext}[1]{\textcolor{green}{#1}}
\nc{\redtext}[1]{\textcolor{red}{#1}}
\nc{\bluetext}[1]{\textcolor{blue}{#1}}
\nc{\brown}[1]{\browntext{ #1}}
\nc{\green}[1]{\greentext{ #1}}
\nc{\red}[1]{\redtext{ #1}}
\nc{\blue}[1]{\bluetext{ #1}}
\nc{\zb}[1]{\redtext{From zb: #1}}
\usepackage{colordvi}
\newcommand{\TT}{\operatorname{\texttt{\rm T}}\nolimits}

\newtheorem{theorem}{Theorem}  [section]
\newtheorem{corollary}[theorem]{Corollary}
\newtheorem{lemma}[theorem]{Lemma}
\newtheorem{proposition}[theorem]{Proposition}

\newtheorem{definition}[theorem]{Definition}
\newtheorem{example}[theorem]{Example}

\theoremstyle{remark}
\newtheorem{remark}[theorem]{Remark}

\numberwithin{equation}{section}

\newtheorem{innercustomthm}{{\bf Theorem}}
\newenvironment{customthm}[1]
{\renewcommand\theinnercustomthm{#1}\innercustomthm}
{\endinnercustomthm}

\renewcommand{\ker}{\operatorname{Ker}\nolimits}

\renewcommand{\mod}{\operatorname{mod}\nolimits}
\renewcommand{\Im}{\operatorname{Im}\nolimits}
\newcommand{\res}{\operatorname{res}\nolimits}
\newcommand{\RHom}{\operatorname{RHom}\nolimits}
\newcommand{\Aut}{\operatorname{Aut}\nolimits}
\newcommand{\Id}{\operatorname{Id}\nolimits}
\newcommand{\dimv}{\operatorname{\underline{dim}}\nolimits}

\newcommand{\Ext}{\operatorname{Ext}\nolimits}
\newcommand{\coker}{\operatorname{Coker}\nolimits}
\newcommand{\rep}{\operatorname{rep}\nolimits}
\newcommand{\aut}{\operatorname{Aut}\nolimits}
\newcommand{\Hom}{\operatorname{Hom}\nolimits}
\newcommand{\Iso}{\operatorname{Iso}\nolimits}
\newcommand{\Fac}{\operatorname{Fac}\nolimits}
\newcommand{\add}{\operatorname{add}\nolimits}

\def \cb{{\mathcal B}}
\def \ce{{\mathcal E}}
\def \cI{{\mathcal I}}

\def \bS{{\mathbf S}}

\def \bl{{\mathbf l}}
\def \bm{{\mathbf m}}
\def \bn{{\mathbf n}}
\def \ba{{\mathbf a}}

\def \cs{{\mathcal S}}

\def \cu{{\mathcal U}}
\def \cv{{\mathcal V}}

\def \y{{B}}

\newcommand{\coh}{\operatorname{coh}\nolimits}

\newcommand{\iH}{{}^{\imath}\widetilde\ch}

\newcommand{\tMHL}{{}^\imath\widetilde{\ch}(\bfk \QK)}
\newcommand{\tMHLJ}{{}^\imath\widetilde{\ch}(\bfk \QJ)}

\newcommand{\pd}{\operatorname{proj.dim}\nolimits}
\newcommand{\ind}{\operatorname{inj.dim}\nolimits}

\newcommand{\ch}{{\mathcal H}}
\newcommand{\ca}{{\mathcal A}}

\newcommand{\calc}{{\mathcal C}}

\newcommand{\mbf}{\mathbf}

\newcommand{\mrm}{\mathrm}

\newcommand{\End}{\mrm{End}}

\newcommand{\de}{\delta}

\def \co{{\mathcal O}}
\newcommand{\haH}{\widehat{H}}
\newcommand{\haT}{\widehat{\Theta}}
\newcommand{\haP}{\widehat{P}}

\renewcommand{\P}{\mathbb P}
\newcommand{\PL}{\bbP^1_{\bfk}}

\newcommand{\LaJ}{\Lambda_{\texttt{J}}}
\newcommand{\LaK}{\Lambda_{\texttt{Kr}}}
\newcommand{\N}{\mathbb N}
\newcommand{\bbZ}{\mathbb Z}

\newcommand{\ov}{\overline}
\newcommand{\qbinom}[2]{\begin{bmatrix} #1\\#2 \end{bmatrix} }
\newcommand{\Q}{\mathbb Q}
\newcommand{\QJ}{Q_{{\texttt{J}}}}
\newcommand{\QK}{Q_{{\texttt{Kr}}}}
\newcommand{\sll}{\mathfrak{sl}}

\newcommand{\K}{\mathbb K}
\newcommand{\F}{\mathbb F}

\newcommand{\bs}{\mathbf s}
\newcommand{\arxiv}[1]{\href{http://arxiv.org/abs/#1}{\tt arXiv:\nolinkurl{#1}}}

\newcommand{\Z}{\mathbb Z}

\newcommand{\B}{\mbf V}
\newcommand{\D}{\mbf D}

\def \bS{{\mathbf S}}
\def \cd{{\mathcal D}}
\def \ck{{\mathcal K}}
\def \cs{{\mathcal S}}

\newcommand{\tK}{\mathbb K}

\def \ct{{\mathcal T}}

\def \BG{{\mathbb G}}
\def \BS{{\mathbb S}}
\def \BF{{\digamma}}

\newcommand{\tUi}{\widetilde{{\mathbf U}}^\imath}
\newcommand{\sqq}{{\bf v}}

\newcommand{\tMHX}{\operatorname{{}^\imath\widetilde{\ch}(\PL)}\nolimits}
\newcommand{\tCMH}{\operatorname{{}^\imath\widetilde{\calc}(\bfk \QK)}\nolimits}
\newcommand{\tCMHP}{\operatorname{{}^\imath\widetilde{\calc}(\PL)}\nolimits}

\def\co{{\mathcal O}}

\def\coh{{\rm coh}}
\def\vec{{\rm vec}}
\def\tor{{\rm tor}}
\def\rk{{\rm rk}}
\def\bfk{\Bbbk}

\def\bbZ{{\mathbb Z}}

\def\bbP{{\mathbb P}}

\def\bbF{{\mathbb F}}

\newcommand{\tUiD}{\operatorname{{}^{{Dr}}\tUi}\nolimits}
\begin{document}
	\title[$\imath$Hall algebra of the projective line and $q$-Onsager algebra]{$\imath$Hall algebra of the projective line \\and $q$-Onsager algebra}
	
	\author[Ming Lu]{Ming Lu}
	\address{Department of Mathematics, Sichuan University, Chengdu 610064, P.R. China}
	\email{luming@scu.edu.cn}

	\author[Shiquan Ruan]{Shiquan Ruan}
	\address{ School of Mathematical Sciences,
		Xiamen University, Xiamen 361005, P.R. China}
	\email{sqruan@xmu.edu.cn}
	
	\author[Weiqiang Wang]{Weiqiang Wang}
	\address{Department of Mathematics, University of Virginia, Charlottesville, VA 22904, USA}
	\email{ww9c@virginia.edu}

	
	\begin{abstract}
		The $\imath$Hall algebra of the projective line is by definition the twisted semi-derived Ringel-Hall algebra of the category of $1$-periodic complexes of coherent sheaves on the projective line. This $\imath$Hall algebra is shown to realize the universal $q$-Onsager algebra (i.e., $\imath$quantum group of split affine $A_1$ type) in its Drinfeld type presentation. The $\imath$Hall algebra of the Kronecker quiver was known earlier to realize the same algebra in its Serre type presentation. We then establish a derived equivalence which induces an isomorphism of these two $\imath$Hall algebras, explaining the isomorphism of the $q$-Onsager algebra under the two presentations.
	\end{abstract}
	
	\maketitle
	\setcounter{tocdepth}{1}
	\tableofcontents

	\section{Introduction}

	\subsection{}
	
	Bridgeland \cite{Br13} has realized a whole quantum group via the Hall algebra of $2$-periodic complexes \cite{Br13}, building on the classic construction of Ringel-Hall algebra of a quiver which realizes half a quantum group \cite{Rin90, Lus90, Gr95}.
	
	Recently, two of the authors \cite{LW20, LW22} have developed $\imath$Hall algebras of $\imath$quivers to realize the (universal) quasi-split $\imath$quantum groups of Kac-Moody type. A universal $\imath$quantum group admits a Serre type presentation and contains various central generators, which replace the parameters in an $\imath$quantum group arising from quantum symmetric pairs \`a la G.~Letzter \cite{Let02, Ko14}. The $\imath$Hall algebras are constructed in the framework of semi-derived Ringel-Hall algebras of 1-Gorenstein algebras (\cite[Appendix A]{LW22}, \cite{LW20}), which were generalizations of earlier constructions \cite{Br13, Gor13, LP21}. In particular, Bridgeland's Hall algebras realization of quantum groups can be reformulated as $\imath$Hall algebras of $\imath$quivers of diagonal type.
	
	There has been a current realization of the affine quantum groups formulated by Drinfeld \cite{Dr88, Be94, Da15}, which plays a crucial role on (algebraic and geometric) representation theory. Hall algebra of the projective line was studied in a visionary paper by Kapranov \cite{Ka97} and then extended by Baumann-Kassel \cite{BKa01} to realize the current half of quantum affine $\sll_2$. The Hall algebra of a weighted projective line was developed in \cite{Sch04} to realize half an affine quantum group of ADE type, which were then upgraded to the whole quantum group via Drinfeld double techniques  \cite{DJX12, BS13}.
	
	According to the $\imath$program philosophy \cite{BW18}, $\imath$quantum groups are viewed as a vast generalization of quantum groups, and various (algebraic, geometric, categorical) constructions on quantum groups should be generalizable to $\imath$quantum groups. Earlier notable examples of such generalizations include $q$-Schur duality, (quasi-) R-matrix, canonical basis, Hall algebras, and quiver varieties.
	
	As a most recent development in the $\imath$program, a Drinfeld type (or current) realization of the (universal) $\imath$quantum group of split affine ADE type has been obtained by two of the authors \cite{LW21b}. The (universal) $\imath$quantum group of split affine $A_1$ type is also known as the (universal) $q$-Onsager algebra $\tUi$. The current presentation in the rank one case was motivated by the construction of root vectors in \cite{BK20}, where one finds more references on the $q$-Onsager algebra.

	\subsection{}
	
	The goal of this paper is to realize the universal $q$-Onsager algebra in its current presentation $\tUiD$  via the $\imath$Hall algebra of the projective line over a finite field $\bfk = \bbF_q$, denoted by  $\tMHX$. (By $\imath$Hall algebra of the projective line, we mean  the twisted semi-derived Ringel-Hall algebra of the category of $1$-periodic complexes of coherent sheaves on the projective line. Both this category and the category of modules of a $\imath$quiver algebra are weakly 1-Gorenstein exact categories, and so the general machinery of semi-derived Ringel-Hall algebra in \cite[Appendix A]{LW22} applies.) 
	We further show that the isomorphism of the universal $q$-Onsager algebra in two (Serre vs Drinfeld) presentations is induced from a derived equivalence of the categories underlying the two (quiver vs $\PL$) $\imath$Hall algebra realizations.

	\subsection{}
	
	In its current presentation \cite{LW21b}, the universal $q$-Onsager algebra $\tUiD$ is generated by $B_{1,r}, H_m$, for $r\in \Z, m\ge 1$, and two central elements $\K, C$, subject to relations \eqref{iDR1}--\eqref{iDR3}. The generators $H_m$ can be replaced by another set of generators $\Theta_m$, for $m\ge 1$, and the relations \eqref{iDR1}--\eqref{iDR2} can be replaced by \eqref{eq:hh1}--\eqref{eq:hB1}.
	
	The following is the first main result of this paper. Let $\sqq =\sqrt{q}$.
	
	\begin{customthm}{\bf A} [Theorems~\ref{thm:morphi}, Proposition~\ref{prop:HaH}]
		There exists a $\Q(\sqq)$-algebra homomorphism
		$\Omega: \tUiD_{|_{v=\sqq}} \longrightarrow \tMHX$
		which sends, for all $r\in \Z$ and $m \ge 1$,
		\begin{align*}
			\K_1\mapsto [K_\co],  \quad
			C\mapsto [K_\de],  \quad
			B_{1,r} \mapsto -\frac{1}{q-1}[\co(r)],  \quad
			\Theta_{m}\mapsto  \haT_m, \quad
			H_m \mapsto \haH_m.
		\end{align*}
	\end{customthm}
	The imaginary root vectors $\Theta_m$ (and $\haH_m$) are realized via $\haT_m$ in terms of torsion sheaves on $\PL$; see \eqref{def:Theta} and \eqref{formula for Hm}, while the real root vectors $B_{1,r}$ are realized via torsion-free sheaves $\co(r)$.
	
	To show that $\Omega: \tUiD_{|_{v=\sqq}} \rightarrow \tMHX$ is a homomorphism, we must verify three main relations for the $q$-Onsager algebra, \eqref{iDR3}, \eqref{eq:hh1}, and \eqref{eq:hB1}.
	
	The counterpart in $\tMHX$ of the relation \eqref{eq:hh1} asserts the commutativity among $\haT_m$, where $\haT_m$ are defined in terms of torsion sheaves. It is well known that the category of torsion sheaves supported on a fixed closed point $x \in \PL$ is equivalent to the category of finite-dimensional nilpotent representations of the Jordan quiver (i.e., the quiver with one vertex and one loop).
	As we show the $\imath$Hall algebra of the Jordan quiver is commutative, the commutativity of $\haT_m$ follows.
	
	Recall the Hall algebra of the Jordan quiver is historically the original example of Hall's construction; it is isomorphic to the ring of symmetric functions and leads to a basis given by Hall-Littlewood functions. The $\imath$Hall algebra of the Jordan quiver admits rich combinatorial properties as well, which will be studied in depth in a separate publication
	\cite{LRW21}.

	Various well-known constructions in the Hall algebra of the Jordan quiver are essentially used in computations of the counterpart of $\haT_m$ in the Hall algebra of $\PL$ in \cite[Example~ 4.12]{Sch12}. (This did not appear in \cite{BKa01}.) As we do not have the results on $\imath$Hall algebra of the Jordan quiver available to us, our verification of the relation \eqref{eq:hB1} takes a more direct approach which requires some serious computations. In addition, the relations \eqref{eq:hB1} and \eqref{iDR3} contain terms involving $\K_\de$ which do not arise in the computations \cite{Ka97, BKa01, Sch12} of similar identities in Hall algebra of the projective line; some new homological computations are needed to determine these $\K_\de$ terms.
	
	Finally, we obtain an $\imath$Hall algebra realization $\haH_m$ of the generators $H_m$ in $\tUi$; cf. Proposition~\ref{prop:HaH}. In contrast to its counterpart in \cite{BKa01} (also cf. \cite{Sch12}), the $\haH_m$ has a subtle summand involving $K_\de$. While we have a self-contained long proof for the formula of $\haH_m$, a similar proof is still needed to produce a similar result for $\imath$Hall algebra of the Jordan quiver \cite{LRW21}. On the other hand, appealing to this result {\em loc. cit.} allows us to shorten the proof for $\haH_m$ considerably, which is the approach we follow here.

	\subsection{}
	
	Beilinson \cite{Bei79} constructed a tilting object which induces a derived equivalence $\cd^b(\coh(\PL))\stackrel{\simeq}{\rightarrow} \cd^b(\rep_\bfk( \QK))$, where $\rep_\bfk( \QK)$ is the category of finite-dimensional representations of the Kronecker quiver $\QK$; see \eqref{Bei}. We establish a similar derived equivalence in the setting of 1-periodic complexes (which is again induced by a tilting object), cf. Proposition~\ref{prop:tilt}. 
	
	As a special case of the main result in \cite{LW20}, there is a realization of the universal $q$-Onsager algebra $\tUi$ (in its Serre type presentation) via the $\imath$Hall algebra of the Kronecker quiver $\tMHL$, that is, we have an injective homomorphism $\widetilde\psi: \tUi_{ |v={\sqq}} \rightarrow \tMHL$. The aforementioned derived equivalence induces an isomorphism of $\imath$Hall algebras $\tMHL \cong \tMHX$, providing a categorification of the algebra isomorphism $\tUi \cong \tUiD$ of the $q$-Onsager algebra in two (Serre and Drinfeld type) presentations. We summarize our second main result as follows.
	
	\begin{customthm}{\bf B} [Proposition~\ref{prop:F}, Theorem~\ref{main thm2}]
		We have the following commutative diagram 
		\[
		\xymatrix{\tUi_{ |v={\sqq}} \ar[d]^{\widetilde{\psi}} \ar[r]^{\cong}   & {}^{\text{Dr}}\tUi_{ |v=\sqq} \ar[d]^{\Omega}
			\\
			\tMHL \ar[r]^{\cong} & \tMHX  }\]
		In particular, the homomorphism $\Omega$ is injective.
	\end{customthm}
	
	This is similar to the interpretation by Burban-Schiffmann \cite{BS12} that the Drinfeld-Beck isomorphism for quantum affine $\sll_2$ 
	can be explained via Beilinson's derived equivalence when combined with Cramer's result \cite{Cr10}. A possible relevance of derived equivalence to the two different presentations of quantum affine $\sll_2$ was remarked by Kapranov \cite{Ka97} (also cf. \cite{BKa01}).

	\subsection{}
	This work opens up further research directions.
	It will be natural to develop connections between $\imath$Hall algebras of weighted projective lines and the $\imath$quantum groups of split affine ADE type in Drinfeld type current presentations, and this will be carried out in \cite{LR21}. It will also be interesting to study the $\imath$Hall algebras of higher genus curves, in particular, of elliptic curves.

	\subsection{}
	
	The paper is organized as follows.
	In Section~\ref{sec:HallPL}, we review the category of coherent sheaves on the projective line and define the corresponding $\imath$Hall algebra. The new Drinfeld type presentation of the universal $q$-Onsager algebra is summarized in Section~\ref{sec:Onsager}.
	
	In Section~\ref{sec:HallOn}, we show that $\Omega: \tUiD_{|_{v=\sqq}} \rightarrow \tMHX$ is a homomorphism in Theorem~{\bf A} by verifying the three defining relations of $\tUiD$ in $\tMHX$.
	A derived equivalence leading to the isomorphism of $\imath$Hall algebras is established and then Theorem~{\bf B} is proved in Section~\ref{sec:derived}.
	In Section~\ref{sec:T}, we provide an $\imath$Hall algebra realization $\haH_m$ of the generators $H_m$.

	\vspace{2mm}
	\noindent {\bf Acknowledgement.}
	ML thanks University of Virginia, Shanghai Key Laboratory of Pure Mathematics and Mathematical Practice, East China Normal University for hospitality and support. ML is partially supported by the Science and Technology Commission of Shanghai Municipality (grant No. 18dz2271000), and the National Natural Science Foundation of China (grant No. 12171333).
	SR is partially supported by the National Natural Science Foundation of China (grant No. 11801473) and the Fundamental Research Funds for Central Universities of China (grant No. 20720220043).
	WW is partially supported by the NSF grant DMS-1702254 and DMS-2001351.

	\section{$\imath$Hall algebra of the projective line}
	\label{sec:HallPL}

	In this section, we review some basic facts on the category $\coh(\PL)$ of coherent sheaves of the projective line over a finite field $\Bbbk$ (also cf. \cite{BKa01}). We then apply the machinery of semi-derived Ringel-Hall algebra \cite{LP21, LW22, LW20} to formulate the $\imath$Hall algebra of $\coh(\PL)$ (and also the $\imath$Hall algebras of the Jordan quiver and the Kronecker quiver).

	\subsection{Coherent sheaves on $\PL$}
	
	Let $\bfk=\bbF_{q}$ be a finite field of $q$ elements. For a (not necessarily acyclic) quiver $Q$, we denote
	
	$\triangleright$ $\rep_\bfk(Q)$ -- category of finite-dimensional representations of $Q$ over $\bfk$,
	
	$\triangleright$ $\rep^{\rm nil}_\bfk(Q)$ -- subcategory of $\rep_\bfk(Q)$ consisting of nilpotent representations of $Q$.
	
	\noindent Note that $\rep_{\bfk}(Q)= \rep^{\rm nil}_\bfk(Q)$ if $Q$ is acyclic.
	
	For a quiver with relations $(Q,I)$, let $\Lambda=\bfk Q/(I)$ be its (not necessarily finite-dimensional) quiver algebra. We define $\rep_\bfk(Q,I)$ and $\rep_\bfk^{\rm nil}(Q,I)$ similarly. We also denote  $\rep(\Lambda)\stackrel{\rm def}{=}\rep_\bfk(Q,I)$ and $\rep^{\rm nil}(\Lambda)\stackrel{\rm def}{=}\rep^{\rm nil}_\bfk(Q,I)$. Note that $\rep(\Lambda)=\rep^{\rm nil}(\Lambda)$ if $\Lambda$ is finite-dimensional.

	%
	The coordinate ring of the projective line $\PL$ over $\bfk$ is the $\bbZ$-graded ring $\bS=\bfk[X_0,X_1]$ with $\deg(X_0)=\deg(X_1)=1$. A closed point $x$ of $\PL$ is given by a prime homogeneous ideal of $\bS$ generated by an irreducible polynomial in $\bS$. The degree of $x$, denoted by $\deg(x)$ or $d_x$, is defined to be the degree of the defining irreducible polynomial associated to $x$.
	We denote
	
	$\triangleright$
	$\mod^{\bbZ}(\bS)$ -- category of finitely generated $\bbZ$-graded $\bS$-modules,
	
	$\triangleright$
	$\mod_0^{\bbZ}(\bS)$-- the full subcategory of $\mod^{\bbZ}(\bS)$ of finite-dimensional graded $\bS$-modules,
	
	$\triangleright$
	$\coh(\PL)$ -- category of cohorent sheaves on $\PL$,
	
	$\triangleright$
	$\vec(\PL)$ -- category of locally free sheaves on $\PL$,
	
	$\triangleright$
	$\tor(\PL)$ -- category of torsion sheaves on $\PL$,
	
	$\triangleright$
	$\tor_x(\PL)$ -- category of torsion sheaves on $\PL$ supported on a closed point $x \in \PL$.
	
	We can associate a coherent sheaf $\widetilde{M}$ on $\PL$ to any $M \in \mod^{\bbZ}(\bS)$. This gives rise to a category equivalence    (which goes back to Serre):
	$\mod^{\mathbb{Z}}(\bS) \big / \mbox{mod}_{0}^{\mathbb{Z}}(\bS)
	\cong \coh(\PL).$
	The category $\coh(\PL)$ is a finitary hereditary abelian Krull-Schmidt category with Serre duality of the form
	\[
	\Ext^1(X, Y)\cong{\rm D}\Hom(Y, \tau(X)),
	\]
	where $D=\Hom_\bfk(-,\bfk)$ and $\tau$ is given by the grading shift with $(-2)$. This implies the existence of almost split sequences for $\coh(\PL)$ with the Auslander-Reiten translation $\tau$.
	
	The pair $(\tor(\PL), \vec(\PL))$ forms a split torsion pair in $\coh(\PL)$, namely, any coherent sheaf can be decomposed as a direct sum of a torsion sheaf and a vector bundle, and there are no nonzero homomorphisms from $\tor(\PL)$ to $\vec(\PL)$.
	
	Any indecomposable vector bundle on $\PL$ is a line bundle; more precisely \cite{Gro57}, it is of the form $\co(n)=\widetilde{\bS[n]}$, for $n\in\bbZ$, where $\bS[n]$ is the $n$-th shift of the trivial module $\bS$, i.e., $\bS[n]_i=\bS_{n+i}$. In particular, if $n=0$, then $\co:=\co(0)$ is the structure sheaf of $\PL$.
	The homomorphism between two line bundles are given by
	\begin{align}   \label{eq:OmOn}
		\Hom(\co(m),\co(n))\cong \bS_{n-m},
	\end{align}
	which then has dimension $n-m+1$ if $n\geq m$ and $0$ otherwise.
	
	The category of torsion sheaves splits into a direct sum of blocks
	\[
	\tor(\PL)=\bigoplus\limits_{x\in\PL}\tor_x(\PL).
	\]
	The category $\tor_x(\PL)$ is equivalent to the category of finite-dimensional modules over the discrete valuation ring (stalk) $\co_x$ at $x$. Hence, it is equivalent to the category $\rep^{\rm{nil}}_{\bfk_x} (\QJ)$ of finite-dimensional nilpotent representations of the Jordan quiver $\QJ$ over the residue field $\bfk_x$ of $\co_x$, where $\bfk_x$ is a finite field extension of $\bfk$ with $[\bfk_x:\bfk]=d_x$. 
	Any indecomposable object in $\tor_x(\PL)$ is of the form $S_x^{(n)}$ of length $n\geq 1$, where $S_x =S_x^{(1)}$ is simple.
	
	The $S_x^{(n)}$ is uniserial in the sense that all the subobjects of $S_x^{(n)}$ form an increasing chain
	\[
	\xymatrix{0\ar[r]& S_{x}\ar[r]^{u}&S_{x}^{(2)}\ar[r]^{u}& \cdots\ar[r]^{u} &S_{x}^{(n-1)}\ar[r]^{u}&S_{x}^{(n)},}
	\]
	and all the quotient objects of $S_i^{(n)}$ form a decreasing chain
	\[
	\xymatrix{S_{x}^{(n)}\ar[r]^{p}& S_{x}^{(n-1)}\ar[r]^{p}& \cdots\ar[r]^{p} &S_{x}^{(2)}\ar[r]^{p}&S_{x} \ar[r]^{p}& 0.}
	\]
	Here $u: S_{x}^{(i-1)}\rightarrow S_{x}^{(i)}$ denotes the irreducible injection map and $p: S_{x}^{(i)}\rightarrow S_{x}^{(i-1)}$ denotes the irreducible surjection map, respectively. 
	
	Assume that a closed point $x$ is determined by a homogeneous polynomial $f(X_1, X_2)$ of degree $d$. Then $S_x^{(n)}$ is determined by the exact sequence
	\[
	0 \longrightarrow  \co(-nd)  \xrightarrow{f(X_1,X_2)^n}
	\co \longrightarrow S_x^{(n)} \longrightarrow 0.
	\]
	In particular, we have
	\begin{align}   \label{hom:SS}
		\Hom(\co, S_x^{(n)})\cong\bfk_x^n,
		\qquad
		\Hom(S_x^{(m)}, S_x^{(n)})\cong { D}\Ext^1(S_x^{(n)}, S_x^{(m)})\cong\bfk_x^{\min\{m,n\}}.
	\end{align}
	For $m\leq n$, the extension in $\Ext^1(S_x^{(n)}, S_x^{(m)})$ is of the following form, for some $0\leq a\leq m$:
	\begin{align}  \label{ext:SS}
		\xymatrix{0\ar[r]& S_x^{(m)}\ar[rr]^-{(p^{a}, u^{n-m+a})^t}&&S_x^{(m-a)}\oplus S_x^{(n+a)}\ar[rr]^-{(u^{n-m+a}, -p^{a})}&&S_x^{(n)}\ar[r]&0.}
	\end{align}
	
	We denote by $\mathbb N (\PL)$ the set of all functions ${\bf n}: \PL \rightarrow \N$ such that ${\bf n}_x \neq 0$ for only finitely many $x\in \PL$. We sometimes write $\bn \in \mathbb N (\PL)$ as ${\bf n} =({\bf n}_x)_{x\in \PL}$ or  ${\bf n} =({\bf n}_x)_x$.  We define a partial order $\leq$ on $\mathbb N (\PL)$:
	\begin{align}  \label{eq:mn}
		\bn \leq \bm \text{ if and only if } \bn_x \le \bm_x \text{ for all }x\in \PL.
	\end{align}
	For $\bn \in \mathbb N (\PL)$, we denote the torsion sheaf
	\begin{align}  \label{eq:Sbn}
		S_\bn =\bigoplus\limits_{x\in\PL}S_{x}^{(\bn_x)},
	\end{align}
	whose degree is given by
	\begin{align*}
		|| {\bf n} || :=\sum_{x\in \bbP^1_{\mathbf{k}}} d_x {\bf n}_x.
	\end{align*}
	
	For two distinct closed points $x,y\in \PL$, the categories $\tor_x(\PL)$ and $\tor_y(\PL)$ are orthogonal in the sense that
	there are no nontrivial homomorphisms and extensions between them, that is, $\Hom(S_x^{(m)}, S_y^{(n)})=0=\Ext^1(S_x^{(m)}, S_y^{(n)})$, for $m,n\geq 1$.
	
	Let $T=\co\oplus \co(1)$. Then \cite{Bei79} $T$ is a tilting sheaf in $\coh(\PL)$, whose endomorphism ring is the quiver algebra $\bfk \QK$ of the Kronecker quiver
	\[
	\QK:\xymatrix{ 0\ar@<0.5ex>[r] \ar@<-0.5ex>[r]&1 }
	\]
	It follows that there is an equivalence between the bounded derived categories
	\begin{align*}
		\cd^b(\coh(\PL))\cong \cd^b(\rep_\bfk(\QK)).
	\end{align*}
	Denote by $\widehat{\mathcal F}$ the image in its Grothendieck group $K_0(\ca)$ of the isoclass of $\mathcal F$ in an abelian category $\ca$. Then, the isomorphism classes $\widehat{\co}$ and $\widehat{\co(1)}$ form a basis of $K_0(\PL):= K_0(\coh(\PL))$. Denote by
	\begin{align}
		\delta:=\widehat{\co(1)}-\widehat{\co},
		\label{eq:delta}
	\end{align}
	then $\{\widehat{\co}, \delta\}$ is also a basis. We define two $\bbZ$-linear functions degree and rank on $K_0(\PL)$ such that
	\[
	\deg(\widehat{\co})=0,
	\quad \deg(\delta)=1,
	\quad \rk(\widehat{\co})=1,
	\quad \rk(\delta)=0.
	\]
	Then $\deg(\widehat{\co(d)})=d$ and $\deg(\widehat{S_x^{(n)}})=n{d_x}$ for $d\in \Z$ and $n>0$; and for any coherent sheaf $F$, the integer $\rk(\widehat{F})$ coincides with the geometric rank of $F$.
	Moreover, the assignment
	\[
	K_0(\PL)\longrightarrow \bbZ^2, \qquad
	\widehat{F}\mapsto (\rk(\widehat{F}), \deg(\widehat{F})),
	\]
	is an isomorphism of $\bbZ$-modules.
	
	The Euler form on $K_0(\PL)$ is defined and then computed by Riemann-Roch as follows: 
	\begin{align}
		\begin{split}
			\langle \widehat{E},\widehat{F}\rangle
			&=\dim_{\bfk}\Hom(E,F)-\dim_{\bfk}\Ext^{1}(E,F)
			\\
			&=\rk(\widehat{E})\rk(\widehat{F})-
			\rk(\widehat{E})\deg(\widehat{F})+\rk(\widehat{F})\deg(\widehat{E}).
		\end{split}
		\label{eq:Euler}
	\end{align}

	\subsection{Hall algebras}
	\label{subsec:HA}
	Let $\ce$ be an essentially small exact category in the sense of Quillen, linear over a finite field $\bfk=\F_q$. Assume that $\ce$ is Hom-finite and $\Ext^1$-finite. 
	Given objects $M,N,L\in\ce$, let $\Ext^1(M,N)_L\subseteq \Ext^1(M,N)$ be the subset parameterizing extensions whose middle term is isomorphic to $L$. The {\em Hall algebra} (or {\em Ringel-Hall algebra}) $\ch(\ce)$ is defined to be the $\Q$-vector space with the isoclasses $[M]$ of objects $M$ of $\ce$ as a basis and multiplication given by (cf., e.g., \cite{Br13})
	\[
	[M]\diamond [N]=\sum_{[L]\in \Iso(\ce)}\frac{|\Ext^1(M,N)_L|}{|\Hom(M,N)|}[L].
	\]
	
	\begin{remark}
		Ringel's version of Hall algebra \cite{Rin90} uses a different multiplication formula, but these two versions of Hall algebra are isomorphic by rescaling the generators by the orders of automorphisms.
	\end{remark}
	
	Given three objects $X,Y,Z$, the Hall number is defined to be
	$$F_{XY}^Z:= |\{L\subseteq Z\mid L \cong Y\text{ and }Z/L\cong X\}|.$$
	Denote by $\aut(X)$ the automorphism group of $X$. The Riedtman-Peng formula reads
	\begin{align}
		\label{eq:RP}
		F_{XY}^Z= \frac{|\Ext^1(X,Y)_Z|}{|\Hom(X,Y)|} \cdot \frac{|\aut(Z)|}{|\aut(X)| |\aut(Y)|}.
	\end{align}

	\subsection{Category of $1$-periodic complexes}
	\label{subsec:periodic}
	
	Let $\ca$ be a hereditary abelian category  which is essentially small with finite-dimensional homomorphism and extension spaces. 
	
	A $1$-periodic complex $X^\bullet$ in $\ca$ is a pair $(X,d)$ with $X\in\ca$ and a differential $d:X\rightarrow X$. A morphism $(X,d) \rightarrow (Y,e)$ is given by a morphism $f:X\rightarrow Y$ in $\ca$ satisfying $f\circ d=e\circ f$. Let $\calc_1(\ca)$ be the category of all $1$-periodic complexes in $\ca$. Then $\calc_1(\ca)$ is an abelian category.
	A $1$-periodic complex $X^\bullet=(X,d)$ is called acyclic if $\ker d=\Im d$. We denote by $\calc_{1,ac}(\ca)$ the full subcategory of $\calc_1(\ca)$ consisting of acyclic complexes.
	Denote by $H(X^\bullet)\in\ca$ the cohomology group of $X^\bullet$, i.e., $H(X^\bullet)=\ker d/\Im d$, where $d$ is the differential of $X^\bullet$.
	
	The category $\calc_1(\ca)$ is Frobenius with respect to the degreewise split exact structure.
	The $1$-periodic homotopy category $\ck_1(\ca)$ is obtained as the stabilization
	of $\calc_1(\ca)$, and the $1$-periodic derived category $\cd_1(\ca)$ is the localization of the homotopy category $\ck_1(\ca)$ with respect to quasi-isomorphisms. Both $\ck_1(\ca)$ and $\cd_1(\ca)$ are triangulated categories.
	
	Let $\calc^b(\ca)$ be the category of bounded complexes over $\ca$ and $\cd^b(\ca)$ be the corresponding derived category with the shift functor $\Sigma$.  Then there is a covering functor $\pi:\calc^b(\ca)\longrightarrow \calc_1(\ca)$, inducing a covering functor $\pi:\cd^b(\ca)\longrightarrow \cd_1(\ca)$ which is dense (see, e.g.,
	\cite[Lemma 5.1]{St17}). The orbit category $\cd^b(\ca)/\Sigma$ is a triangulated category \cite{Ke2}, and we have
	\begin{align}
		\label{dereq}
		\cd_1(\ca)\simeq \cd^b(\ca)/\Sigma.
	\end{align}
	
	For any $X\in\ca$, denote the stalk complex by
	\[
	C_X =(X,0)
	\]
	(or just by $X$ when there is no confusion), and denote by $K_X$ the following acylic complex:
	\[
	K_X:=(X\oplus X, d),
	\qquad \text{ where }
	d=\left(\begin{array}{cc} 0&\Id \\ 0&0\end{array}\right).
	\]
	
	\begin{lemma}
		[also cf. \cite{LinP}]
		\label{lem:pd acyclic}
		For any acyclic complex $K^\bullet$ and $p \ge 2$, we have
		\begin{align}
			\label{Ext2vanish}
			\Ext^p_{\calc_1(\ca)}(K^\bullet,-)=0=\Ext^p_{\calc_1(\ca)}(-,K^\bullet).
		\end{align}
	\end{lemma}

	\begin{proof}
		It is enough to prove $\Ext^p_{\calc_1(\ca)}(K^\bullet,C_X)=0=\Ext^p_{\calc_1(\ca)}(C_X,K^\bullet)$ for any $X\in\ca$, by noting that $\calc_1(\ca)$ coincides with the extension closure of $C_X$ ($X\in\ca$).
		
		Denote by $\calc_m(\ca)$ the category of $m$-periodic complexes over $\ca$ for any $m\geq1$.
		By the same proof of \cite[Proposition 2.3]{LP21}, one can obtain that the analogous result holds for $m$-periodic acyclic complexes ($m\geq2$).
		
		Fix $m$ below such that ${\rm char}\bfk\nmid m$. There is a covering functor
		\begin{align*}
			\pi_*: \calc_{m}(\ca)\longrightarrow \calc_1(\ca)
		\end{align*}
		which admits a left (and also right) adjoint functor $\pi^*: \calc_1(\ca)\longrightarrow \calc_m(\ca)$ preserving acyclic complexes. One can prove that \eqref{Ext2vanish} holds for any acyclic complex $K^\bullet\in\Im (\pi_*)$.
		Since ${\rm char}\bfk\nmid m$, the adjunction $ K^\bullet\longrightarrow\pi_*\pi^*(K^\bullet)$ induces that $K^\bullet$ is a direct summand of $\pi_*\pi^*(K^\bullet)$. So \eqref{Ext2vanish} holds for any acyclic complex $K^\bullet$.
	\end{proof}
	
	\begin{lemma}
		\label{lem: iso in singularity}
		For any $X^\bullet,Y^\bullet\in\calc_1(\ca)$, we have $H(X^\bullet)=H(Y^\bullet)$ if and only if
		there exist two short exact sequences
		\begin{align*}
			0\longrightarrow U_1^\bullet\longrightarrow Z^\bullet \longrightarrow X^\bullet\longrightarrow0,
			\qquad 0\longrightarrow U_2^\bullet\longrightarrow Z^\bullet\longrightarrow Y^\bullet\longrightarrow0
		\end{align*}
		with $U^\bullet_1,U^\bullet_2\in \calc_{1,ac}(\ca)$. 
	\end{lemma}
	
	\begin{proof}
		We only need to prove the ``only if'' part.
		If $H(X^\bullet)\cong H(Y^\bullet)$, then $X^\bullet\cong Y^\bullet$ in $\cd_1(\ca)$ by \eqref{dereq}. The desired exact sequences follow from the definition of $\cd_1(\ca)$.
	\end{proof}
	
	Let $\cb$ be an abelian category. For any $B\in\cb$, its Ext-projective dimension $\pd B$ is defined to be the smallest number $i\in\N$ such that
	$\Ext^{i+1}_{\cb}(B,-)=0$; dually one can define its Ext-injective dimension $\ind B$.
	
	\begin{corollary}
		\label{cor:acyclic}
		For any $K^\bullet\in\calc_1(\ca)$ the following are equivalent: (i) $\pd K^\bullet<\infty$;
		(ii)~ $\ind K^\bullet<\infty$;
		(iii) $\pd K^\bullet\leq1$;
		(iv) $\ind K^\bullet\leq1$;
		(v) $K^\bullet$ is acyclic.
	\end{corollary}
	
	\begin{proof}
		The proof is the same as that of \cite[Corollary 2.12]{LW20}, now with the help of Lemma~ \ref{lem:pd acyclic} and Lemma \ref{lem: iso in singularity}.
	\end{proof}

	\begin{remark}
		\label{rem:weakly}
		For any hereditary abelian category $\ca$, it follows from Corollary \ref{cor:acyclic} that $\calc_1(\ca)$ is a weakly $1$-Gorenstein exact category. Therefore the general machinery of semi-derived Ringel-Hall algebras in \cite[Appendix A]{LW22} will be applicable to $\calc_1(\ca)$.
	\end{remark}

	\subsection{$\imath$Hall algebras}
	\label{subsec:iHall}
	
	We continue to work with a hereditary abelian category $\ca$ as in \S\ref{subsec:periodic}.
	Let $\ch(\calc_1(\ca))$ be the Ringel-Hall algebra of $\calc_1(\ca)$ over $\Q(\sqq)$, i.e., $\ch(\calc_1(\ca))=\bigoplus_{[X^\bullet]\in \Iso(\calc_1(\ca))} \Q(\sqq)[X^\bullet]$, with multiplication defined by
	\begin{align*}
		[M^\bullet]\diamond[N^\bullet]=\sum_{[L^\bullet]\in\Iso(\calc_1(\ca)) } \frac{|\Ext^1(M^\bullet,N^\bullet)_{L^\bullet}|}{|\Hom(M^\bullet,N^\bullet)|}[L^\bullet].
	\end{align*}
	Following \cite{LP21, LW22, LW20, LinP}, we consider the ideal $\cI$ of $\ch(\calc_1(\ca))$ generated by
	\begin{align}
		\label{eq:ideal}
		&\{[K_1^\bullet]-[K_2^\bullet] \mid K_1^\bullet,K_2^\bullet\in\calc_{1,ac}(\ca) \text{ with }\widehat{\Im d_{K_1^\bullet}}=\widehat{\Im d_{K_2^\bullet}}\} \bigcup
		\\\notag
		&\{[L^\bullet]-[K^\bullet\oplus M^\bullet]\mid \exists \text{ exact sequence } 0 \rightarrow K^\bullet \rightarrow L^\bullet \rightarrow M^\bullet \rightarrow 0 \text{ with }K^\bullet \text{ acyclic}\}.
	\end{align}
	We denote
	\[
	\cs:=\{ a[K^\bullet] \in \ch(\calc_1(\ca))/\cI \mid a\in \Q(\sqq)^\times, K^\bullet \text{ acyclic}\},
	\]
	a multiplicatively closed subset of $\ch(\calc_1(\ca))/ \cI$ with the identity $[0]$.

With the help of Corollary \ref{cor:acyclic} and Remark \ref{rem:weakly}, we have the following.
	
	\begin{lemma}
		[\text{\cite[Proposition A.5]{LW22}}]
		The multiplicatively closed subset $\cs$ is a right Ore, right reversible subset of $\ch(\calc_1(\ca))/\cI$. Equivalently, there exists the right localization of
		$\ch(\calc_1(\ca))/\cI$ with respect to $\cs$, denoted by $(\ch(\calc_1(\ca))/\cI)[\cs^{-1}]$.
	\end{lemma}

	The algebra $(\ch(\calc_1(\ca))/\cI)[\cs^{-1}]$ is the {\em semi-derived Ringel-Hall algebra} of $\calc_1(\ca)$ in the sense of \cite{LP21, LW22} (also cf. \cite{Gor13}), and will be denoted by $\cs\cd\ch(\calc_1(\ca))$.
	
	For any $\alpha\in K_0(\ca)$,  there exist $X,Y\in\ca$ such that $\alpha=\widehat{X}-\widehat{Y}$. Define $[K_\alpha]:=[K_X]\diamond [K_Y]^{-1}$. This is well defined, see, e.g.,
	\cite[\S 3.2]{LP21}. Denote by $\ct(\ca)$ the subalgebra of $\cs\cd\ch(\calc_1(\ca))$ generated by all acyclic complexes $[K^\bullet]$.
	
	
	
	
	The following lemma is well known.
	\begin{lemma}
		\label{lem:Ext=Ext+Hom}
		For  $X, Y \in \ca$, we have
		\begin{align*}
			\Ext^1_{\calc_1(\ca)}(C_X,C_Y)\cong\Ext^1_\ca(X,Y)\oplus \Hom_\ca(X,Y).
		\end{align*}
	\end{lemma}
	
	For any $K^\bullet\in\calc_{1,ac}(\ca)$ and $M^\bullet\in\calc_{1}(\ca)$, by Corollary \ref{cor:acyclic}, define
	\begin{align*}
		\langle K^\bullet, M^\bullet\rangle=\dim_{\bfk} \Hom_{\calc_{1}(\ca)}( K^\bullet, M^\bullet)-\dim_{\bfk}\Ext^1_{\calc_{1}(\ca)}(K^\bullet, M^\bullet),
		\\
		\langle  M^\bullet,K^\bullet\rangle=\dim_{\bfk} \Hom_{\calc_{1}(\ca)}( M^\bullet, K^\bullet)-\dim_{\bfk}\Ext^1_{\calc_{1}(\ca)}( M^\bullet,K^\bullet).
	\end{align*}
	These formulas give rise to well-defined bilinear forms (called {\em Euler forms}), again denoted by $\langle \cdot, \cdot \rangle$, on the Grothendieck groups $K_0(\calc_{1,ac}(\ca))$
	and $K_0(\calc_{1}(\ca))$.
	
	Denote by $\langle \cdot,\cdot\rangle_\ca$ the Euler form of $\ca$.
	Let $\res: \calc_1(\ca)\rightarrow\ca$ be the restriction functor.
	Then we have the following.
	\begin{lemma}
		\label{lemma compatible of Euler form}
		We have
		\begin{itemize}
			\item[(1)]
			$\langle K_X, M^\bullet\rangle = \langle X,\res (M^\bullet) \rangle_\ca$,\; $\langle M^\bullet,K_X\rangle =\langle \res(M^\bullet), X \rangle_\ca$, for $X\in \ca$, $M^\bullet\in\calc_1(\ca)$;
			\item[(2)] $\langle M^\bullet,N^\bullet\rangle=\frac{1}{2}\langle \res(M^\bullet),\res(N^\bullet)\rangle_\ca$, for  $M^\bullet,N^\bullet\in\calc_{1,ac}(\ca)$.
		\end{itemize}
	\end{lemma}
	
	\begin{proof}
		The proof is the same as for \cite[Proposition 2.4]{LP21} and \cite[Lemma 4.3]{LW22}, hence omitted here.
	\end{proof}
	
	Define
	\[
	\sqq :=\sqrt{q}.
	\]
	
	\begin{definition}  \label{def:iH}
		The {\em $\imath$Hall algebra} of a hereditary abelian category $\ca$, denoted by $\iH(\ca)$,  is defined to be the twisted semi-derived Ringel-Hall algebra of $\calc_1(\ca)$, that is, the $\Q(\sqq)$-algebra on the same vector space as $\cs\cd\ch(\calc_1(\ca)) =(\ch(\calc_1(\ca))/\cI)[\cs^{-1}]$ equipped with the following modified multiplication (twisted via the restriction functor $\res: \calc_1(\ca)\rightarrow\ca$)
		\begin{align}
			\label{eq:tH}
			[M^\bullet]* [N^\bullet] =\sqq^{\langle \res(M^\bullet),\res(N^\bullet)\rangle_\ca} [M^\bullet]\diamond[N^\bullet].
		\end{align}
	\end{definition}
	For any complex $M^\bullet$ and acyclic complex $K^\bullet$, we have
	\[
	[K^\bullet]*[M^\bullet]=[K^\bullet\oplus M^\bullet]=[M^\bullet]*[ K^\bullet].
	\]
	It follows that $[K_\alpha]\; (\alpha\in K_0(\ca))$ are central in the algebra $\iH(\ca)$.
	
	
	The {\em quantum torus} $\widetilde{\ct}(\ca)$ is defined to be the subalgebra of $\iH(\ca)$ generated by $[K_\alpha]$, for $\alpha\in K_0(\ca)$. 
	
	\begin{proposition}  [cf. \cite{LW22,LW20}]
		\label{prop:hallbasis}
		The folllowing hold in $\iH(\ca)$:
		\begin{enumerate}
			\item
			The quantum torus $\widetilde{\ct}(\ca)$ is a central subalgebra of $\iH(\ca)$.
			\item
			The algebra $\widetilde{\ct}(\ca)$ is isomorphic to the group algebra of the abelian group $K_0(\ca)$.
			\item
			$\iH(\ca)$ has an ($\imath$Hall) basis given by
			\begin{align*}
				\{[M]*[K_\alpha]\mid [M]\in\Iso(\ca), \alpha\in K_0(\ca)\}.
			\end{align*}
		\end{enumerate}
	\end{proposition}
	
	\begin{proof}
		Part (1)  has been proved above.
		The proof of (3) is the same as \cite[Theorem~ 3.6]{LW20}, and hence omitted here. Part (2) follows from (3).
	\end{proof}
	
	For any $f:X\rightarrow Y$ in $\ca$, we denote by
	\begin{align*}
		C_f:=\Big( Y\oplus X,  \begin{pmatrix} 0&f \\0&0  \end{pmatrix}  \Big)\in\calc_1(\ca).
	\end{align*}
	
	\begin{lemma}
		\label{lem:Cf}
		For any $M^\bullet=(M,d)$, we have $[M^\bullet]=[H(M^\bullet)]*[K_{\Im d}]$ in $\iH(\ca)$. In particular, for any $f:X\rightarrow Y$, we have
		\begin{align*}
			[C_f]= [\ker f\oplus \coker f]*[K_{\Im f}].
		\end{align*}
	\end{lemma}
	
	\begin{proof}
		By \eqref{eq:ideal}, if $M^\bullet$ is acyclic, then we have $[M^\bullet]=[K_{\Im d}]=[K_{\widehat{\Im d}}]$.
		For general $M^\bullet$, note that $M^\bullet \cong H(M^\bullet)$ in $\cd_1(\ca)$. By Lemma \ref{lem: iso in singularity}, we have the following two
		exact sequences
		\begin{align*}
			0\longrightarrow U_1^\bullet\longrightarrow Z^\bullet \longrightarrow H(M^\bullet)\longrightarrow0,
			\qquad 0\longrightarrow U_2^\bullet\longrightarrow Z^\bullet\longrightarrow M^\bullet\longrightarrow0
		\end{align*}
		with $U^\bullet_1,U^\bullet_2\in \calc_{1,ac}(\ca)$.
		Similar to \cite[Lemma 3.12]{LP21}, we have
		$$\widehat{\Im d_{U_1^\bullet}}= \widehat{\Im d_{Z^\bullet}}=\widehat{\Im d_{U_2^\bullet}}+\widehat{\Im d}.$$
		Then
		\begin{align*}
			[M^\bullet]=&[Z^\bullet]*[U_2^\bullet]^{-1}
			=[H(M^\bullet)]* [ U_1^\bullet]*[U_2^\bullet]^{-1}
			\\
			=&[H(M^\bullet)]*[ K_{\Im d_{U_1^\bullet}}]*[K_{\Im d_{U_2^\bullet}}]^{-1}
			\\
			=&[H(M^\bullet)\oplus K_{\Im d}].
		\end{align*}
		The lemma is proved.
	\end{proof}
	
	In the remainder of this paper we will study the $\imath$Hall algebras of the hereditary abelian categories $\ca$ in the following example, which are intimately related to each other.

	\begin{example}
		\label{ex:Q}
		\quad
		\begin{enumerate}
			\item
			$\ca= \coh(\PL)$.
			\item
			Let  $Q$ be a quiver. Recall the $\imath$quiver algebra $\Lambda^\imath$ associated to a split $\imath$quiver $(Q,\Id)$ from \cite{LW22, LW20} is given by $\Lambda^\imath =\Bbbk Q \otimes \Bbbk [\varepsilon]/(\varepsilon^2)$. Take $\ca =\rep^{\rm nil}_\bfk(Q)$. The category $\calc_1(\rep^{\rm nil}_\bfk(Q))$ 
			can be naturally identified with the category $\rep^{\rm nil}(\Lambda^\imath)$. 
			We shall specialize the quiver to the Jordan quiver $\QJ$ in \S\ref{subsec:Jordan} and then the Kronecker quiver $\QK$ in Section~\ref{sec:derived}.
		\end{enumerate}
	\end{example}

	The $\imath$Hall algebras for the $\imath$quiver algebras in Example \ref{ex:Q}(2) were studied in depth in \cite{LW22, LW20}. To study the $\imath$Hall algebra  for $\coh(\P^1_\bfk)$, we shall need the preparatory results in \S\ref{subsec:periodic}-\S\ref{subsec:iHall}.

	\section{The $q$-Onsager algebra and its current presentation}
	\label{sec:Onsager}
	
	In this section, we review the (universal) $q$-Onsager algebra and its Drinfeld type presentation from \cite{LW20}.

	\subsection{The $q$-Onsager algebra}
	\label{subsec:On}
	
	For $n\in \Z, r\in \N$, denote by
	\[
	[n] =\frac{v^n -v^{-n}}{v-v^{-1}},\qquad
	\qbinom{n}{r} =\frac{[n][n-1]\ldots [n-r+1]}{[r]!}.
	\]
	For $A, B$ in a $\Q(v)$-algebra, we shall denote $[A,B]_{v^a} =AB -v^aBA$, and $[A,B] =AB - BA$.
	
	Recall \cite{LW22, LW20} (compare \cite{BK20}) that  the (universal) $q$-Onsager algebra $\tUi$ is a $\Q(v)$-algebra with unity generated by $B_0,B_1$, $\K_0^{\pm1},\K_1^{\pm1}$, subject to the following (Serre type) relations:
	\begin{align}
		&\K_i\K_i^{-1}=1,  \qquad \K_i \text{ are central}, \qquad i=0,1;
		\\
		&\sum_{r=0}^3 (-1)^r \qbinom{3}{r}B_i^{3-r} B_j B_i^{r}= - v^{-1} [2]^2 (B_iB_j-B_jB_i) \K_i, \quad \text{ if } i\neq j.
	\end{align}
	(In \cite{LW22}, $\widetilde{k}_i$ are used in place of $\K_i$, and they are related by $\K_i=-v^2\widetilde{k}_i$, for $i=0,1$. The $\K_i$ are directly related to $\imath$Hall algebra.) The $q$-Onsager algebra $\tUi$ is the $\imath$quantum group of split affine $A_1$ type, a special case of $\imath$quantum groups in \cite{LW20, LW21b}.
	
	Let $\Z\alpha_0\oplus\Z\alpha_1$ be the root lattice of affine $\mathfrak{sl}_2$. Let $\delta:=\alpha_0+\alpha_1$.
	For any $\beta=a_0\alpha_0+a_1\alpha_1$, define
	$
	\tK_\beta=\K_0^{a_0}\K_1^{a_1}.
	$ 
	In particular, we have
	\[
	\K_\delta =\K_0 \K_1.
	\]
	($\K_\de$ will often be denoted by $C$ later on.)

	Let $\dag$ be the involution of the $\Q(v)$-algebra $\tUi$ such that
	\begin{align}
		\dag:B_0\leftrightarrow B_1, \quad \K_0\leftrightarrow \K_1.
	\end{align}
	We have the following two automorphisms $\TT_0,\TT_1$ \cite{LW21b}, which admit an interpretation in $\imath$Hall algebras (see \cite{LW21a} and its forthcoming sequel):
	\begin{align}
		\TT_1 (\K_1) &=\K_1^{-1},\qquad \TT_1(\K_0)= \K_{\delta} \K_1,
		\\
		\TT_1(B_1)&=  \K_1^{-1} B_1,
		\\
		\TT_1(B_0)&=  [2]^{-1} \big(B_0B_1^{2} -v[2] B_1 B_0B_1 +v^2 B_1^{2} B_0 \big) + B_0\K_1,
		\label{T1B0}
		\\
		\TT_1^{-1}(B_0)&=  [2]^{-1} \big( B_1^{2}B_0-v[2] B_1B_0B_1 +v^2 B_0B_1^{2} \big) +B_0\K_1.
		\label{T1B0-2}
	\end{align}
	%
	%
	The action of $\TT_0$ is obtained from the above formulas by switching indices $0,1$, that is,
	\begin{align}
		\TT_0=\dag \circ \TT_1 \circ \dag.
	\end{align}
	
	For any $r, m\in\Z$, modifying \cite{BK20} as in \cite{LW21b}, we define
	\begin{align}
		B_{1,r} &=(\dag \TT_1)^{-r}(B_1),
		\label{eq:B1n} \\
		\acute{\Theta}_{m} &=
		\begin{cases}
			-B_{1,m-1} B_0+v^{2} B_0B_{1,m-1} + (v^{2}-1)\sum_{p=0}^{m-2} B_{1,p} B_{1,m-p-2} \K_0, &\text{ if }m>0,
			\\
			\frac{1}{v-v^{-1}}, & \text{ if }m=0,
			\\
			0,& \text{ if }m<0.
		\end{cases}
		\label{eq:dB1}
	\end{align}
	Note that $B_{1,0}=B_1$ by definition. 

	For any $m\in\Z$, we define $\Theta_{m}$ recursively such that
	(see \cite{LW21b})
	\begin{align}
		\Theta_m=\begin{cases}
			v^{-2} \K_{\de}\Theta_{m-2}  +\acute{\Theta}_{m}- \K_{\de}\acute{\Theta}_{m-2},   & \text{ if }m>0,
			\\
			\frac{1}{v-v^{-1}}, & \text{ if }m=0,
			\\
			0,&\text{ if }m<0.
		\end{cases}
	\end{align}
	Note that $\Theta_1=\acute{\Theta}_1$, and $\Theta_{2}=\acute{\Theta}_{2}-v^{-1}\K_{\delta}$.
	As emphasized {\em loc. cit.}, the definition of $\Theta_m$ is motivated by the study of $\imath$Hall algebra of coherent sheaves of $\PL$ in this paper.
	
	
	%

	%
	%
	\subsection{A Drinfeld type presentation of $\tUi$}
	
	\begin{definition} [\cite{LW21b}]
		Let $\tUiD$ be the $\Q(v)$-algebra  generated by $\K_1^{\pm1}$, $C^{\pm1}$, $H_{m}$ and $\y_{1,r}$, where $m\geq1$, $r\in\Z$, subject to the following relations, for $r,s\in \Z$ and $m,n\ge 1$:
		\begin{align}
			\K_1\K_1^{-1}=1, C C^{-1}& =1, \quad \K_1, C \text{ are central, }
			\\
			[H_m,H_n] &=0,  \label{iDR1}
			\\
			[H_m, \y_{1,r}] &=\frac{[2m]}{m} \y_{1,r+m}-\frac{[2m]}{m} \y_{1,r-m}C^m,
			\label{iDR2}
			\\
			\label{iDR3}
			[\y_{1,r}, \y_{1,s+1}]_{v^{-2}}  -v^{-2} [\y_{1,r+1}, \y_{1,s}]_{v^{2}}
			&= v^{-2}\Theta_{s-r+1} C^r \K_1-v^{-4} \Theta_{s-r-1} C^{r+1} \K_1 \\\notag
			&\quad +v^{-2}\Theta_{r-s+1} C^s \K_1-v^{-4} \Theta_{r-s-1} C^{s+1} \K_1.\notag
		\end{align}
		Here
		\begin{align}
			\label{eq:exp1}
			1+ \sum_{m\geq 1} (v-v^{-1})\Theta_{m} z^m  = \exp\Big( (v-v^{-1}) \sum_{m\ge 1}  H_m z^m \Big).
		\end{align}
	\end{definition}
	The relations \eqref{iDR1}--\eqref{iDR2} in the presentation $\tUiD$ can be replaced by \eqref{eq:hh1}--\eqref{eq:hB1} below.
	
	\begin{lemma}
		[\cite{LW21b}]
		\label{lem:equiv}
		\quad
		\begin{enumerate}
			\item
			The relation~\eqref{iDR1} $(\text{for } m, n\geq 1)$  is equivalent to
			\begin{align}
				\label{eq:hh1}
				[\Theta_{m},\Theta_n] = 0 \quad (m, n\geq 1).
			\end{align}
			\item
			The relation \eqref{iDR2} $(\text{for } r\in \Z, m\geq 1)$  is equivalent to
			\begin{align}
				\label{eq:hB1}
				&[\Theta_{m}, \y_{1,r}]+[\Theta_{m-2},\y_{1,r}]C\\
				&=v^{2}[\Theta_{m-1},\y_{1,r+1}]_{v^{-4}}+v^{-2}[\Theta_{m-1},\y_{1,r-1}]_{v^{4}}C \quad (r\in\Z, m\geq 1).
				\notag
			\end{align}
		\end{enumerate}
	\end{lemma}
	
	\begin{theorem}
		[\cite{LW21b}]
		\label{thm:UUiso}
		Let $\tUi$ be the $q$-Onsager algebra. Then there is an isomorphism of $\Q(v)$-algebras
		$ \tUiD \rightarrow\tUi$
		such that
		\[
		B_{1,r}\mapsto B_{1,r}, \quad
		\Theta_m \mapsto \Theta_m, \quad
		\K_1 \mapsto \K_1,  \quad
		C\mapsto \K_\de
		\quad (r \in \Z, m \ge 1).
		\]
	\end{theorem}

	\section{$\imath$Hall algebra and $q$-Onsager algebra}
	\label{sec:HallOn}
	
	In this section, we establish a homomorphism from the $q$-Onsager algebra in its Drinfeld type presentation to the $\imath$Hall algebra of the projective line.

	\subsection{The homomorphism $\Omega$ }
	
	We shall use a shorthand notation $\tMHX$ to denote the $\imath$Hall algebra $\iH (\coh(\P^1_\bfk))$, cf. Definition~\ref{def:iH}.
	Recalling $S_\bn \in \tor(\PL)$ from \eqref{eq:Sbn}, we introduce the following elements in $\tMHX$:
	\begin{align}
		\label{def:Theta}
		\haT_{m}= \frac{1}{(q-1)\sqq^{m-1}} \sum\limits_{||\bn||=m}[S_\bn],
		\qquad \text{ for } m\ge 1.
	\end{align}
	We also set
	\[\haT_0=\frac{1}{\sqq-\sqq^{-1}},\qquad \haT_m=0, \,\forall m<0.\]
	%

	Here is another description of $\haT_m$.
	\begin{lemma}
		For $s\in \Z$ and $m\ge 1$, we have
		\begin{align}  \label{eq:hTm}
			\haT_{m}= \frac{1}{(q-1)^2\sqq^{m-1}}\sum_{0\neq f:\co(s)\rightarrow \co(m+s) } [\coker f].
		\end{align}
	\end{lemma}
	
	\begin{proof}
		Without loss of generality, we shall only prove the case for $s=0$.
		Any nonzero morphism $f:\co\rightarrow \co(m)$ is given by a homogeneous polynomial in $\bS$, which can be decomposed as a product of irreducible polynomials $f=\prod\limits_{x\in \PL} f_x^{\bm_x}$, for some $\bm\in\N (\PL)$ with $||\bm ||=m$, where $f_x$ denotes the irreducible polynomial in $\bS$ corresponding to the closed point $x$. Hence $\coker f\cong S_{\bm}$. Moreover, for any morphism $g:\co\rightarrow \co(m)$, we have $\coker f\cong \coker g$ if and only if $g=\mu f$ for some nonzero $\mu\in\bfk$. Therefore,
		\[
		\sum_{0\neq f:\co\rightarrow \co(m) } [\coker f]=(q-1) \sum_{||\bm||=m} [ S_{\bm}].
		\]
		This proves the lemma.
	\end{proof}
	
	The goal of this section is to prove the following theorem.
	\begin{theorem}
		\label{thm:morphi}
		There exists a $\Q(\sqq)$-algebra homomorphism
		\begin{align}
			\label{eq:phi}
			\Omega: \tUiD_{|_{v=\sqq}} \longrightarrow \tMHX
		\end{align}
		which sends, for all $r\in \Z$ and $m \ge 1$,
		\begin{align*}
			\K_1\mapsto [K_\co],  \quad
			C\mapsto [K_\de],  \quad
			B_{1,r} \mapsto -\frac{1}{q-1}[\co(r)],  \quad
			\Theta_{m}\mapsto  \haT_m.
		\end{align*}
	\end{theorem}
	We shall verify that Relations~\eqref{iDR3}, \eqref{eq:hh1} and \eqref{eq:hB1} are preserved by $\Omega$, thanks to Lemma~\ref{lem:equiv}.
	(Later in Theorem~\ref{main thm2}, we shall strengthen Theorem ~\ref{thm:morphi} by showing that $\Omega$ is injective.)

	\subsection{Relation \eqref{iDR3}}
	
	The relation \eqref{iDR3} in $\tMHX$ is formulated as identities \eqref{HallRR}--\eqref{HallRR1} among $\co(r)$ in Proposition~\ref{prop:OO} below. To that end, we shall first compute the product among $\co(r)$ for various $r$.
	
	\begin{lemma}
		For $r\in\Z$ and $m\in\N$, the following identities hold in $\tMHX$:
		\begin{align}
			[\co(r)]*[\co(r)] 
			&=\sqq^{-1}[\co(r)\oplus \co(r)] +\sqq^{-1}(q-1)[K_{\co(r)}], \label{rr1}
			\\
			[\co(r)]*[\co(r+m+1)]
			&= \sqq^{-(m+2)} [\co(r)\oplus \co(r+m+1)] + \frac{(q-1)^2}{\sqq^2} \haT_{m+1}*[K_{\co(r)}],\label{rr2}
			\\
			[\co(r+m+1)]*[\co(r)]
			&= \sqq^{-m} [\co(r)\oplus \co(r+m+1)]
			\label{rr3}\\
			+ \sum_{1\leq a\leq \lfloor\frac{m+1}{2}\rfloor; a\neq \frac{m+1}{2}}(\sqq^4 -&1 )  \sqq^{4a-4-m}  [\co(r+a)\oplus \co(r+m+1-a)]
			\notag
			\\
			+ \delta_{m, odd}(\sqq^2-1)\sqq^{m-2}  \Big[&\co(r+\frac{m+1}{2})\oplus \co(r+\frac{m+1}{2}) \Big].
			\notag
		\end{align}
	\end{lemma}
	
	\begin{proof}
		Note that any nonzero morphism $f:\co(r)\rightarrow \co(s)$, for $r\leq s$, is injective.
		
		The identity \eqref{rr1} follows since $\Ext^1(\co(r),\co(r))=0$ and $\dim_\bfk\Hom(\co(r),\co(r))=1$.
		
		Note that $\Ext^1(\co(r),\co(r+m+1))=0$, and $\dim_\bfk\Hom(\co(r),\co(r+m+1))=m+2$. From \eqref{eq:OmOn} we have the following formula for the Euler form:
		\begin{align}  \label{eq:formOO}
			\langle \co(r),\co(s)\rangle=s-r+1, \quad \text{ for } r,s \in \Z.
		\end{align}
		Now the identity \eqref{rr2} is obtained by the following computation using \eqref{eq:tH} and \eqref{eq:formOO}:
		\begin{align*}
			&[\co(r)]*[\co(r+m+1)]
			\\
			&=\sqq^{\langle \co(r),\co(r+m+1)\rangle} [\co(r)]\diamond[\co(r+m+1)]
			\\
			&=\sqq^{m+2} [\co(r)]\diamond[\co(r+m+1)]
			\\
			&=
			\sqq^{-(m+2)}[\co(r)\oplus \co(r+m+1)] + \sqq^{-(m+2)}  \sum_{0\neq f: \co(r)\rightarrow \co(r+m+1)} [C_f]
			\\
			&=\sqq^{-(m+2)}[\co(r)\oplus \co(r+m+1)] + \sqq^{-(m+2)} \sum_{0\neq f:\co(r)\rightarrow \co(r+m+1)} [\coker f]*[K_{\co(r)}]
			\\
			&=\sqq^{-(m+2)}[\co(r)\oplus \co(r+m+1)] + \sqq^{-2}(q-1)^2 \haT_{m+1}*[K_{\co(r)}],
		\end{align*}
		where the last equality uses \eqref{eq:hTm}.
		
		The computation for \eqref{rr3} can be performed in the setting of $\coh(\PL)$, thanks to Lemma~ \ref{lem:Ext=Ext+Hom} and $\Hom(\co(r+m+1),\co(r))=0$. Note the following formula, for $m<n$ and $1\leq a\leq \lfloor\frac{n-m}{2}\rfloor$ (cf. \cite{BKa01}):
		\[
		\big|\Ext^1(\co(n),\co(m))_{\co(m+a)\oplus \co(n-a)} \big|=
		\left\{
		\begin{array}{l}
			(\sqq^4-1)\sqq^{4a-4}, \quad \text{ if }m+a\neq n-a\\
			(\sqq^2-1)\sqq^{4a-4}, \quad \text{ if } m+a=n-a.
		\end{array}
		\right.
		\]
		Hence \eqref{rr3} follows since
		\begin{align*}
			[\co(r+m+1)]*[\co(r)]
			&= \sqq^{-m}[\co(r+m+1)]\diamond[\co(r)]
			\\
			&= \sqq^{-m}[\co(r)\oplus \co(r+m+1)]
			\\
			&+ \sqq^{-m}\sum_{1\leq a\leq \lfloor\frac{m+1}{2}\rfloor; a\neq \frac{m+1}{2}} (\sqq^4-1)\sqq^{4a-4} [\co(r+a)\oplus \co(r+m+1-a)]
			\notag
			\\
			&+ \delta_{m, odd}(\sqq^2-1)\sqq^{m-2} \Big[\co(r+\frac{m+1}{2})\oplus \co(r+\frac{m+1}{2}) \Big].
		\end{align*}
		The lemma is proved.
	\end{proof}
	
	\begin{proposition}
		\label{prop:OO}
		For $r\in\Z$ and $m\geq2$, we have
		\begin{align}   \label{HallRR}
			[\co(r),&\co(r+m+1)]_{\sqq^{-2}} -\sqq^{-2}[\co(r+1),\co(r+m)]_{\sqq^{2}}
			\\
			&\qquad\qquad\quad = \sqq^{-2} (q-1)^2 \haT_{m+1}*[K_{\co(r)}]-\sqq^{-4} (q-1)^2\haT_{m-1}* [K_{\co(r+1)}], \notag
			\\
			\label{HallRR1/2}
			[\co(r),&\co(r+2)]_{\sqq^{-2}} -\sqq^{-2}[\co(r+1),\co(r+1)]_{\sqq^{2}}
			\\
			&\qquad\qquad\quad = \sqq^{-2} (q-1)^2 \haT_{2}*[K_{\co(r)}]-\sqq^{-3} (q-1)^2[K_{\co(r+1)}], \notag
			\\
			\label{HallRR1}
			[\co(r),&\co(r+1)]_{\sqq^{-2}}=\sqq^{-2} (q-1)^2 \haT_{1}*[K_{\co(r)}].
		\end{align}
	\end{proposition}
	
	\begin{proof}
		Let us prove \eqref{HallRR}. By \eqref{rr2}--\eqref{rr3}, we have
		\begin{align*}
			[\co(r+1)]*[\co(r+m)]
			=&\sqq^{-m} [\co(r+1)\oplus \co(r+m)] + \frac{(q-1)^2}{\sqq^2} \haT_{m}*[K_{\co(r+1)}].
		\end{align*}
		On the other hand, we have
		\begin{align*}
			\co(r+m)*\co(r+1)
			&= \sqq^{-m+2}[\co(r+1)\oplus \co(r+m)]
			\\
			&+ \sum_{1\leq a\leq \lfloor\frac{m-1}{2}\rfloor; a\neq \frac{m-1}{2}} (\sqq^4-1)\sqq^{4a-2-m} [\co(r+1+a)\oplus \co(r+m-a)]
			\\
			&+ \delta_{m, odd}(\sqq^2-1)\sqq^{m-4} [\co(r+1+\frac{m-1}{2})\oplus \co(r+1+\frac{m-1}{2})].
			\notag
		\end{align*}
		Using the above formulas and \eqref{rr2}--\eqref{rr3} again, we prove \eqref{HallRR} by a direct computation:
		\begin{align*}
			&[\co(r),\co(r+m+1)]_{\sqq^{-2}} -\sqq^{-2}[\co(r+1),\co(r+m)]_{\sqq^{2}}
			\\
			=&\sqq^{-(m+2)} [\co(r)\oplus \co(r+m+1)] + \frac{(q-1)^2}{\sqq^2} \haT_{m+1}*[K_{\co(r)}]- \sqq^{-m-2} [\co(r)\oplus \co(r+m+1)]
			\\
			&-\sum_{1\leq a\leq \lfloor\frac{m+1}{2}\rfloor; a\neq \frac{m+1}{2}}(\sqq^4-1)\sqq^{4a-6-m} [\co(r+a)\oplus \co(r+m+1-a)]
			\notag
			\\
			&- \delta_{m, odd}(\sqq^2-1)\sqq^{m-4} [\co(r+\frac{m+1}{2})\oplus \co(r+\frac{m+1}{2})]
			-\sqq^{-(m+2)} [\co(r+1)\oplus \co(r+m)]
			\\
			&- \frac{(q-1)^2}{\sqq^4} \haT_{m-1}*[K_{\co(r+1)}]
			+\sqq^{-m+2}[\co(r+1)\oplus \co(r+m)]
			\\
			&+ \sum_{1\leq a\leq \lfloor\frac{m-1}{2}\rfloor; a\neq \frac{m-1}{2}} (\sqq^4-1)\sqq^{4a-2-m} [\co(r+1+a)\oplus \co(r+m-a)]
			\notag
			\\
			&+ \delta_{m, odd}(\sqq^2-1)\sqq^{m-4} [\co(r+1+\frac{m-1}{2})\oplus \co(r+1+\frac{m-1}{2})]
			\\
			=&\sqq^{-2} (q-1)^2 \haT_{m+1}*[K_{\co(r)}]-\sqq^{-4} (q-1)^2\haT_{m-1}* [K_{\co(r+1)}].
		\end{align*}
		
		The proofs of the identities \eqref{HallRR1/2} and \eqref{HallRR1} are entirely similar by use of \eqref{rr2}--\eqref{rr3}, and will be skipped.
		%
		%
	\end{proof}

	\subsection{Relation \eqref{eq:hh1} and Jordan $\imath$quiver}
	\label{subsec:Jordan}
	
	The relation \eqref{eq:hh1} in $\tMHX$ is formulated as the commutativity among $\haT_m$ in Proposition~\ref{prop:commTheta} below. As $\haT_m$ is defined via torsion sheaves on $\PL$ in \eqref{def:Theta}, we shall approach the commutativity by establishing the commutativity of the $\imath$Hall algebra of the Jordan quiver, which is isomorphic to the $\imath$Hall algebra of $\tor_x(\PL)$, for any closed point $x\in \PL$.
	
	Let $ \QJ$ be the Jordan quiver, i.e., the quiver with a single vertex $1$ and a single loop $\alpha:1\rightarrow 1$.
	Let $\LaJ^\imath$ be the $\imath$quiver algebra of the $\imath$quiver $\QJ$ equipped with trivial involution \cite{LW22, LW20}, and we can identify
	\[
	\LaJ^\imath=\bfk\ov{\QJ} /( \alpha\varepsilon-\varepsilon \alpha, \varepsilon^2 ),
	\]
	where $\ov{\QJ}$ denotes the following quiver
	\begin{center}\setlength{\unitlength}{0.8mm}
		\vspace{-7mm}
		\begin{equation}
			\begin{picture}(30,6)(0,0)
				\put(-30,-2){ $\ov{\QJ}:$}
				\put(0,-2){\small $1$}
				\put(15.5,-2){.}
				
				\qbezier(-1,1)(-3,4.5)(-6,4)
				\qbezier(-6,4)(-11.5,0)(-6,-4)
				\qbezier(-6,-4)(-3,-4.5)(-1,-1)
				\put(-1,-1){\vector(1,1){0.3}}
				\put(-13,-1){\small $\alpha$}
				\color{purple}
				
				\qbezier(3,1)(5,4.5)(8,4)
				\qbezier(8,4)(13.5,0)(8,-4)
				\qbezier(8,-4)(5,-4.5)(3,-1)
				\put(3,-1){\vector(-1,1){0.3}}
				\put(12.5,-1){\small $\varepsilon$}
			\end{picture}
			\vspace{0.2cm}
		\end{equation}
	\end{center}
	Then $\LaJ^\imath$ is a commutative $\bfk$-algebra. Clearly, $\rep^{\rm nil}(\LaJ^\imath)\cong \calc_1(\rep^{\rm nil}_\bfk (\QJ))$, and we shall identify these two categories below. We can view $\rep^{\rm nil}_\bfk(\QJ)$ naturally as a full subcategory of $\rep^{\rm nil}(\LaJ^\imath)$.
	
	Let $\tMHLJ:=\cs\cd\widetilde{\ch}(\rep^{\rm nil}(\LaJ^\imath))$ be the semi-derived Ringel-Hall algebra of $\LaJ^\imath$, following \cite{LW22, LW20} (also see Section~\ref{sec:HallPL}). Note that $\langle X,Y\rangle=0$ for any $X,Y\in\rep^{\rm nil}_\bfk(\QJ)$. So
	\begin{align*}
		[X^\bullet]*[Y^\bullet]=[X^\bullet]\diamond[Y^\bullet],\,\,\forall X^\bullet, Y^\bullet\in \rep^{\rm nil}(\LaJ^\imath).
	\end{align*}
	We shall study the combinatorial implication of $\tMHLJ$ in depth in \cite{LRW21}, and here we only need the following commutativity property.
	
	\begin{lemma}
		\label{lem:comm Jordan}
		The algebra $\tMHLJ$ is commutative.
	\end{lemma}

	\begin{proof}
		The standard duality functor $D=\Hom_\bfk(-,\bfk):\rep^{\rm nil}(\LaJ^\imath) \stackrel{\simeq}{\rightarrow} \rep^{\rm nil}(\LaJ^\imath)$ is a contravariant exact functor. It is well known that $D(X)\cong X$ for any $X\in\rep^{\rm nil}_\bfk (\QJ)$.
		Let $L, M, X\in \rep^{\rm nil}_\bfk(\QJ)$ and $N\in\rep^{\rm nil}(\LaJ^\imath)$. We have
		\begin{align}  \label{eq:EE}
			\Ext^1_{\LaJ^\imath}(L,M)_N \cong \Ext^1_{\LaJ^\imath} \big(D(M),D(L) \big)_{D(N)} \cong  \Ext^1_{\LaJ^\imath}(M,L)_{D(N)}.
		\end{align}
		Let $S$ be the simple $\bfk \QJ$-module.
		By Proposition~ \ref{prop:hallbasis}, $\tMHLJ$ has a Hall basis $\{[X]*[K_{S}]^a\mid X\in \rep^{\rm nil}_\bfk(\QJ), a\in\Z\}$.

		{\bf Claim ($\star$). } We have $[D(N)]=[X\oplus K_S^{\oplus a}]$ if and only if $[N] =[X\oplus K_S^{\oplus a}]$ in $\tMHLJ$.
		
		Let us prove the Claim. By symmetry it suffices to prove the ``if'' part. 
		By Lemma \ref{lem:Cf} and Proposition \ref{prop:hallbasis}, we have $H^\bullet(N)\cong H^\bullet(X)=X$, and then it follows from Lemma \ref{lem: iso in singularity} and Corollary \ref{cor:acyclic} that there exist the following exact sequences in $\rep^{\rm nil}(\LaJ^\imath)$
		\begin{align*}
			0\rightarrow U_1\rightarrow Z \rightarrow N\rightarrow0,
			\qquad 0\rightarrow U_2\rightarrow Z\rightarrow X\rightarrow0
		\end{align*}
		with $U_1,U_2$ acyclic.
		In particular, $N\cong X$ in $\cd_{1}(\rep_\bfk^{\rm nil}(\QJ))$.  
		Hence, we have $D(N) \cong D(X) \cong X$ in $\cd_{1}(\rep_\bfk^{\rm nil}(\QJ))$ since $D$ preserves acyclic complexes. Whence it follows from Lemma \ref{lem:Cf} that $[D(N)]=[X\oplus K_S^{\oplus a}]$ in $\tMHLJ$ by comparing dimensions. The Claim follows.
		
		Let $a\in\N$, $X\in \rep_\bfk^{\rm nil}(\QJ)$. It follows by \eqref{eq:EE} and Claim ($\star$) that there is a natural bijection
		\begin{align*}
			\bigsqcup_{{ [N]=[X\oplus K_S^{\oplus a} ] \in \tMHLJ}} \Ext^1_{\LaJ^\imath}(L,M)_N
			\stackrel{1:1}{\longleftrightarrow}
			\bigsqcup_{
				{ [N]=[X\oplus K_S^{\oplus a} ] \in \tMHLJ}} \Ext^1_{\LaJ^\imath}(M,L)_N.
		\end{align*}
		It is understood that $[N]$ runs over $\Iso \big(\rep^{\rm nil}(\LaJ^\imath) \big)$ in the bijection above.
		
		By $\Hom_{\LaJ^\imath}(L,M)\cong \Hom_{\LaJ^\imath}(M,L)$, we have
		\begin{align*}
			[L]*[M] =& \sum_{[N]\in\Iso(\rep^{\rm nil}(\LaJ^\imath))} 
			\frac{|\Ext^1_{\LaJ^\imath}(L,M)_N|}{|\Hom_{\LaJ^\imath}(L,M)|}    [N]
			\\
			=&\sum_{\stackrel{a\in \N, [X]}{ [N]=[X\oplus K_S^{\oplus a} ] \in \tMHLJ}} \frac{|\Ext^1_{\LaJ^\imath}(L,M)_N|}{|\Hom_{\LaJ^\imath}(L,M)|}    [X]*[K_S]^a
			\\
			=& \sum_{\stackrel{a\in \N, [X]}{ [N]=[X\oplus K_S^{\oplus a} ] \in \tMHLJ}} \frac{|\Ext^1_{\LaJ^\imath}(M,L)_N|}{|\Hom_{\LaJ^\imath}(M,L)|}    [X]*[K_S]^a
			\\
			=&[M]*[L].
		\end{align*}
		It is understood that $[X]$ (and respectively, $[N]$) runs over $\Iso \big(\rep^{\rm nil}_\bfk(\QJ) \big)$ (and respectively, $\Iso \big(\rep^{\rm nil}(\LaJ^\imath) \big)$) in the above summations.
		The lemma is proved.
	\end{proof}
	
	Note however that the (twisted) Hall algebra $\widetilde{\ch}(\calc_1(\rep_\bfk^{\rm nil} (\QJ)))$ is not commutative.

	Recall there exists an equivalence of categories $\tor_x(\PL)\simeq \rep_{\bfk_x}^{\rm nil}(\QJ)$, for each $x\in \PL$, inducing an embedding $\widetilde{\ch}\big(\calc_1(\rep_{\bfk_x}^{\rm nil}(\QJ))\big) \rightarrow \widetilde{\ch}\big(\calc_1(\coh(\PL))\big)$ and then an embedding
	\[
	\iota_{x}:{}^\imath\widetilde{\ch}(\bfk_x\QJ) \longrightarrow \tMHX.
	\]
	
	Denote by ${}^\imath\widetilde{\ch}(\tor(\P_\bfk^1))$ the twisted semi-derived Ringel-Hall algebra of $\calc_1(\tor(\P_\bfk^1))$. Then
	${}^\imath\widetilde{\ch}(\tor(\P_\bfk^1))$ is naturally a subalgebra of $\tMHX$. Note that  $\Ext^1(X,Y)=0=\Hom(X,Y)$ for any $X\in\tor_{x}(\P_\bfk^1), Y \in\tor_y(\P_\bfk^1)$ with $x \neq y\in \PL$.
	Together with Lemma \ref{lem:comm Jordan}, this implies the commutativity among $\haT_m$ below (which is Relation \eqref{eq:hh1} in $\tMHX$).
	
	\begin{proposition}
		\label{prop:commTheta}
		The algebra ${}^\imath\widetilde{\ch}(\tor(\P_\bfk^1))$ is commutative. In particular, we have
		\begin{align*}
			[\haT_{m},\haT_n]=0, \qquad
			\forall m, n \geq 1.
		\end{align*}
	\end{proposition}

	\subsection{Relation \eqref{eq:hB1}}
	
	Recall $q=\sqq^2$. We reformulate the relation \eqref{eq:hB1}  in $\tMHX$ (with $\Theta_m$ replaced by $\haT_m$, $B_{1,r}$ by $-\frac{1}{q-1}[\co(r)]$, and $v$ by $\sqq$) in the following proposition.
	
	\begin{proposition}
		\label{prop:ThetaTheta}
		For any $m\geq 1$ and $r\in\Z$, we have
		\begin{align}  \label{TOTO}
			&\big[\haT_{m},[\co(r)] \big] + \big[\haT_{m-2},[\co(r)] \big]*[K_\de]
			\\\notag
			&=\sqq^{2} \big[\haT_{m-1},[\co(r+1)] \big]_{\sqq^{-4}} +\sqq^{-2} \big[\haT_{m-1},[\co(r-1)] \big]_{\sqq^{4}}*[K_\de].
		\end{align}
	\end{proposition}
	
	The long proof of Proposition~\ref{prop:ThetaTheta} will occupy the remainder of this subsection. We shall compute the products $\haT_{m}*[\co(r)]$ and $[\co(r)]*\haT_{m}$.
	
	%
	\subsubsection{The product $\haT_{m}*[\co(r)]$}
	
	
	Denote by $\Psi(d)$ the number of irreducible polynomials in one variable of degree $d$ over $\F_q$. It is a basic fact in Galois theory of finite fields that
	\begin{equation}   \label{eq:FF}
		\sum_{d |n} d \Psi(d)  =q^n,
		\quad \text{or equivalently,} \sum_{x\in \PL:\, d_x |n} d_x =q^n+1, \quad \text{ for } n\ge 1.
	\end{equation}
	We first prepare some combinatorial formulas, which will be used in the computation of $\haT_{m}*[\co(r)]$ in Proposition~\ref{prop:ThetaO}.
	For a closed point $x\in\PL$, denote
	\[q_x:=q^{d_x},\qquad \sqq_x:=\sqq^{d_x}.\]
	
	\begin{lemma}
		\label{lem:aut}
		For any positive integer $a$, we have
		\begin{align*}  
			\sum\limits_{||\bn||=a} \prod_{\bn_x \ge 1} \big(q_x^{ \bn_x} -q_x^{ \bn_x-1} \big)
			= q^{2a}-q^{2a-2}.
		\end{align*}
	\end{lemma}
	
	\begin{proof}
		For a variable $t$, we note
		\begin{align}  \label{eq:GF1}
			1 +\sum_{a\ge 1} (q^{2a}-q^{2a-2}) t^a
			= 1+ \frac{q^2 t -t}{1-q^2t} = \frac{1-t}{1-q^2t}.
		\end{align}
		On the other hand, we have
		\begin{align}
			&1 + \sum_{a\ge 1} \Big(\sum\limits_{x, \bn_x:||\bn||=a} \prod_{\bn_x \ge 1} (q_x^{ \bn_x} -q_x^{ \bn_x-1}) \Big)
			t^a
			\label{eq:GF2}
			\\
			&= \prod_{x \in \PL} \Big(1+\sum_{\bn_x\ge 1}  (1 -q_x^{-1}) q_x^{ \bn_x} t^{d_x \bn_x} \Big)
			\notag
			\\
			&= \prod_{x \in \PL} \left(1 + \frac{(1-q_x^{-1}) q_xt^{d_x} }{1- q_xt^{d_x}} \right)
			\notag
			\\
			&= \prod_{x \in \PL} \frac{1  - t^{d_x} }{1- q_x^{} t^{d_x}}
			= \frac{1 - t}{1- qt} \cdot \prod_{d =1}^{\infty} \left( \frac{1  - t^{d} }{1- q^{d} t^{d}} \right)^{\Psi(d)}.
			\notag
		\end{align}
		
		
		Thanks to \eqref{eq:GF1}--\eqref{eq:GF2}, the lemma follows from (and is indeed equivalent to) Lemma~\ref{lem:aut2} below.
	\end{proof}
	
	\begin{lemma}
		\label{lem:aut2}
		The following identity holds:
		\begin{align}  \label{eq:prod}
			\prod_{d =1}^{\infty} \left( \frac{1  - t^{d} }{1- q^{d} t^{d}} \right)^{\Psi(d)}
			= \frac{1-qt}{1-q^2t}.
		\end{align}
	\end{lemma}
	
	\begin{proof}
		Using \eqref{eq:FF}, we compute
		\begin{align*}
			\ln \left(\prod_{d =1}^{\infty}  \left( \frac{1  - t^{d} }{1- q^{d} t^{d}} \right)^{\Psi(d)} \right)
			&= \sum_{d=1}^\infty \Psi(d) \big( \ln (1  - t^{d}) -\ln (1- q^{d} t^{d}) \big)
			\\
			&= \sum_{d\ge 1} \Psi(d) \sum_{m\ge 1} \frac1{m} (q^{dm}-1) t^{dm}
			\\
			&\stackrel{(*)}{=} \sum_{n=1}^\infty \sum_{d |n} \Psi(d) d\cdot \frac{1}{n} (q^{n}-1) t^{n}
			\\
			&= \sum_{n=1}^\infty  \frac{1}{n} (q^{2n}-q^n) t^{n}
			= \ln \left( \frac{1-qt}{1-q^2t} \right),
		\end{align*}
		where the identity $(*)$ follows by a change of variables $n=dm$.
		The lemma follows.
	\end{proof}
	
	\begin{lemma}
		\label{lem:claim}
		For  $1\leq a\leq m$ and $\bn \in \N(\PL)$ with  with $||\bn||=m-a$, we have
		\begin{align*}
			\sum\limits_{ ||\bm||=m} \big|\Ext^1( S_{\bm}, \co(r) )_ {\co(r+a)\oplus S_{\bn}} \big| =q^{2a}-q^{2a-2}.
		\end{align*}
	\end{lemma}
	
	\begin{proof}
		If $\big|\Ext^1( S_{\bm}, \co(r) )_ {\co(r+a)\oplus S_{\bn}}\big|\neq0$, then $\bn_x\leq \bm_x$, and so $\bl_x := \bm_x -\bn_x \geq0$, for all $x$; cf. \eqref{eq:mn}. By the Riedtman-Peng formula \eqref{eq:RP}, we have
		\begin{align*}
			&\sum\limits_{ ||\bm||=m} \big|\Ext^1( S_{\bm}, \co(r) )_ {\co(r+a)\oplus S_{\bn}}\big|
			\\
			&=\sum\limits_{ ||\bm||=m}F_{ S_{\bm}, \co(r) }^{\co(r+a)\oplus S_{\bn} } \frac{|\Aut(\co(r))|\cdot |\Aut( S_{\bm})|}{|\Aut(\co(r+a)\oplus S_{\bn} )|}
			\\
			&=\frac{\sum\limits_{ ||\bm||=m} \Big|\{ \co(r+a)\oplus  S_{\bn} \stackrel{f}{\twoheadrightarrow} S_{\bm}\mid \ker f\cong \co(r)\}\Big|}{|\Aut(  S_{\bn})|\cdot|\Hom(\co(r+a), S_{\bn} )|}
			\\
			&=\frac{\sum\limits_{ ||\bm||=m} \Big|\{ \co(r+a)\oplus  S_{\bn} \stackrel{f}{\twoheadrightarrow} S_{\bm}\mid \ker f\cong \co(r)\}\Big|}{|\Aut(  S_{\bn})|\cdot q^{m-a}}.
		\end{align*}
		
		Note that
		\begin{align*}
			&\big|\{ \co  (r+a)\oplus  S_{\bn} \stackrel{f}{\twoheadrightarrow} S_{\bm}\mid \ker f\cong \co(r)\}\big|
			\\
			&=\big|\{(f_{1,x}, f_{2,x})\mid f_{1,x}: \co(r+a)\rightarrow S_{x}^{(\bm_x)} \text{ is onto} \text{ if } \bn_x\neq \bm_x  ,\\
			&\qquad\qquad\qquad f_{2,x}: S_{x}^{(\bn_x)} \rightarrow S_{x}^{(\bm_x)}  \text{ is injective} , \forall x\in \PL \}\big|
			\\
			&= \big|\{f_{1,x}\mid f_{1,x}: \co(r+a)\rightarrow S_{x}^{(\bm_x)} \text{ is onto} \text{ if } \bn_x\neq \bm_x, \forall x\in \PL \}\big|
			\\
			&\quad\cdot \big|\{ f_{2,x}\mid f_{2,x}: S_{x}^{(\bn_x)} \rightarrow S_{x}^{(\bm_x)}  \text{ is injective} , \forall x\in \PL  \}\big|
			\\
			&=\big|\{f_{1,x}\mid f_{1,x}: \co(r+a)\rightarrow S_{x}^{(\bm_x)} \text{ is onto} \text{ if } \bn_x\neq \bm_x, \forall x\in \PL \}\big|\cdot |\Aut(  S_{\bn})|.
		\end{align*}
		So by Lemma \ref{lem:aut}, we have
		\begin{align*}
			&\sum\limits_{ ||\bm||=m}  \big|\Ext^1( S_{\bm}, \co(r) )_ {\co(r+a)\oplus S_{\bn}} \big |
			\\
			&=q^{a-m}\sum\limits_{ ||\bm||=m} \big|\{f_{1,x}: \co(r+a)\rightarrow S_{x}^{(\bm_x)} \mid f_{1,x}\text{ is onto} \text{ if } \bn_x\neq \bm_x\} \big|
			\\
			&=
			q^{a-m}\sum\limits_{\sum\limits_{x}\bl_xd_x=a} \big|\{f_{1,x}: \co(r+a)\rightarrow S_{x}^{(\bm_x)} \mid f_{1,x}\text{ is onto} \text{ if } \bn_x\neq \bm_x\} \big|
			\\
			&=q^{a-m}\sum\limits_{\sum\limits_{x}\bl_xd_x=a} \Big(\prod_{\bl_x\neq 0} (q_x^{ \bl_x+\bn_x} -q_x^{ \bl_x+\bn_x-1} )\cdot \prod_{\bl_x=0} q_x^{\bn_x}\Big)
			\\
			&=\sum\limits_{\sum\limits_{x}\bl_xd_x=a} \prod_{\bl_x\neq 0} (q_x^{ \bl_x} -q_x^{ \bl_x-1} )
			=q^{2a}-q^{2a-2}.
		\end{align*}
		The lemma is proved.
	\end{proof}

	\begin{proposition}
		\label{prop:ThetaO}
		For $r\in \Z$ and $m \ge 1$, we have
		\begin{align}
			\label{prod ThetaO}
			\haT_{m}*[\co(r)]
			=&\frac{1}{(\sqq^2-1)\sqq^{2m-1}}\sum\limits_{||\bm||=m} [\co(r)\oplus  S_{\bm}]\\\notag
			&+\frac{[2]}{\sqq^{2m+2}}\sum\limits_{1\leq a\leq m} \sum\limits_{ ||\bn||=m-a} \sqq^{4a}{[\co(r+a)\oplus  S_{\bn}]}.
		\end{align}
	\end{proposition}

	\begin{proof}
		Since $\big \langle S_{\bm} ,\co(r) \big \rangle=-m$ for any $\bm\in \mathbb N (\PL)$ with $||\bm||=m$, we have
		\begin{align*}
			\haT_{m}*[\co(r)]
			&= \frac{1}{(q-1)\sqq^{m-1}}\sum\limits_{||\bm||=m}
			[S_{\bm}]*[\co(r)]
			\\
			&= \frac{1}{(\sqq^2-1)\sqq^{2m-1}}\sum\limits_{ ||\bm||=m} [S_{\bm}]\diamond[\co(r)].
		\end{align*}
		For any short exact sequence $0\rightarrow \co(r) \rightarrow M\rightarrow  S_{\bm}\rightarrow0$ with $||\bm||=m$, we have
		$$M\cong \co(r)\oplus S_{\bm}, \quad \text{ or }M\cong \co(r+a)\oplus  S_{\bn}, $$
		for some $\bn$ such that $||\bn||=m-a$ and $1\leq a\leq m$. It follows that
		\begin{align*}
			\haT_{m}*[\co(r)]
			=&\frac{1}{(\sqq^2-1)\sqq^{2m-1}} \sum\limits_{ ||\bm||=m} [\co(r)\oplus  S_{\bm}]
			+ \frac{1}{(\sqq^2-1)\sqq^{2m-1}} \sum\limits_{ ||\bm||=m} \sum_{1\leq a\leq r} \\
			&\sum\limits_{ ||\bn||=m-a}
			|\Ext^1( S_{\bm}, \co(r) )_ {\co(r+a)\oplus S_{\bn}}|[\co(r+a)\oplus  S_{\bn}]
			\\
			=&\frac{1}{(\sqq^2-1)\sqq^{2m-1}} \sum\limits_{ ||\bm||=m} [\co(r)\oplus  S_{\bm}]+\frac{1}{(\sqq^2-1)\sqq^{2m-1}} \sum\limits_{1\leq a\leq r}\sum\limits_{ ||\bn||=m-a}
			\\
			&\Big(  \sum\limits_{ ||\bm||=m}|\Ext^1( S_{\bm}, \co(r) )_ {\co(r+a)\oplus S_{\bn}}| \Big)
			{[\co(r+a)\oplus S_{\bn}]}.
		\end{align*}
		Now the proposition follows from applying Lemma~\ref{lem:claim}.
	\end{proof}

	%
	\subsubsection{The product $[\co(r)]*\haT_m$}
	
	\begin{lemma}
		\label{lem:morOS}
		For any morphism $f:\co(r)\rightarrow  S_{\bn}$, there exists $\ba \in \N(\PL)$ with $\ba\leq \bn$ such that
		$$\Im f\cong  S_{\ba}, \quad \coker f\cong S_{\bn-\ba}
		{\quad and \quad}\ker f\cong \co(r-||\ba||).$$
		Moreover, given a surjective morphism $g: \co(r)\rightarrow \Im f$, there exists a unique morphism $f_1: \Im f\rightarrow S_{\bn}$ such that $f=f_1\circ g$.
	\end{lemma}
	
	\begin{proof}
		For a closed point $x$ in $\PL$, the torsion sheaf $S_{x}^{(n)}$ is uniserial, and hence its subsheaf and quotient sheaf are of the form $S_{x}^{(a)}$, for $0\leq a\leq n$. Recall there is no nonzero homomorphism between torsion sheaves supported on distinct closed points. Therefore,  as a subsheaf of $ S_{\bn}$, $\Im f=S_{\ba}$, for some $\ba\leq \bn$.
		It follows that $\coker f\cong S_{\bn-\ba}$. Observe that any subsheaf of a line bundle is again a line bundle, which is determined by its degree. Hence $\ker f\cong \co(r-||\ba||)$.
		The second statement follows from the surjectivity of $g$.
	\end{proof}

	\begin{proposition}
		\label{prop:OTheta}
		For $r\in \Z$ and $m \ge 1$, we have  
		\begin{align*}
			[\co(r)]*\haT_m
			=&\frac{1}{(q-1)\sqq^{2m-1}}\sum\limits_{ ||\bm||=m}[\co(r)\oplus S_{\bm}]\\
			&+ \frac{[2]}{\sqq^{2m+2}}\sum\limits_{1\leq a\leq m}\sum\limits_{ ||\bn||=m-a}
			\sqq^{4a}[\co(r-a)\oplus  S_{\bn}]*[K_{a\de}].
		\end{align*}
	\end{proposition}
	
	\begin{proof}
		For $\bm \in\N (\bbP^1_{\bfk})$ with $||\bm||=m$, we have $\langle \co(r), S_{\bm}\rangle=m$. For any $f\in\Hom(\co(r), S_{\bm})$, we have $\Im f\cong S_{\bl}$ for some $\bl\leq \bm$ with $||\bl||=a$. Then $\ker f\cong \co(r-a)$ and $\coker f\cong S_{\bm-\bl}$. Hence
		\begin{align} \label{eq:Cf}
			[C_f]=[\coker f\oplus \ker f]*[K_{\Im f}]=[\co(r-a)\oplus  S_{\bm-\bl}]*[K_{a\de}].
		\end{align}
		By Lemma \ref{lem:morOS}, we have
		\begin{align}   \label{eq:ff}
			\big|\{ f:\co(r)\rightarrow S_{\bm} \mid \Im f\cong S_{\bl}\} \big|
			=\big|\{ f_1:S_{\bl}\rightarrowtail S_{\bm}\} \big|
			=\prod_{\bl_x\neq 0} \big(q_x^{ \bl_x} -q_x^{ \bl_x-1} \big).
		\end{align}
		Thus using \eqref{eq:Cf}--\eqref{eq:ff} we have
		\begin{align*}
			[\co(r)]*S_{\bm}&=\sqq^{\langle \co(r), S_{\bm}\rangle}\sum\limits_{f\in\Hom(\co(r), S_{\bm})}
			\frac{[C_f]}{|\Hom(\co(r), S_{\bm})|}
			\\
			&=\sqq^{-m}\sum\limits_{0\leq a\leq m}\sum\limits_{\bl\leq \bm;||\bl||=a}\sum\limits_{f;\Im f\cong S_{\bl}}[\co(r-a)\oplus  S_{\bm-\bl}]*[K_{a\de}]
			\\
			&=\sqq^{-m}\sum\limits_{0\leq a\leq m}\sum\limits_{\bl\leq \bm;||\bl||=a}\prod\limits_{\bl_x\geq 1} (q_x^{ \bl_x} -q_x^{ \bl_x-1} )
			[\co(r-a)\oplus  S_{\bm-\bl}]*[K_{a\de}].
		\end{align*}
		Therefore, by Lemma~\ref{lem:aut} we have
		\begin{align*}
			[\co(r)]*\haT_m
			=&\frac{1}{(q-1)\sqq^{m-1}}\sum\limits_{ ||\bm||=m}
			[\co(r)]*[ S_{\bm}]
			\\
			=& \frac{1}{(q-1)\sqq^{2m-1}}\sum\limits_{0\leq a\leq m}\sum\limits_{ ||\bn||=m-a}\sum\limits_{||\bl||=a}\prod_{\bl_x\geq 1} (q_x^{ \bl_x} -q_x^{ \bl_x-1} )
			[\co(r-a)\oplus  S_{\bn}]*[K_{a\de}]\\
			=&\frac{1}{(q-1)\sqq^{2m-1}}\sum\limits_{0\leq a\leq m}\sum\limits_{ ||\bn||=m-a}(q^{2a} -q^{2a-2})[\co(r-a)\oplus  S_{\bn}]*[K_{a\de}]\\
			=&\frac{1}{(q-1)\sqq^{2m-1}}\sum\limits_{ ||\bm||=m}[\co(r)\oplus S_{\bm}]
			\\
			&\qquad + \frac{[2]}{\sqq^{2m+2}}\sum\limits_{1\leq a\leq m}\sum\limits_{ ||\bn||=m-a}
			\sqq^{4a}[\co(r-a)\oplus  S_{\bn}]*[K_{a\de}].
		\end{align*}
		The proposition is proved.
	\end{proof}
	
	%
	
	Now we are ready to prove Proposition \ref{prop:ThetaTheta}, i.e., the identity \eqref{TOTO}.
	
	\begin{proof}[Proof of Proposition \ref{prop:ThetaTheta}]
		%
		%
		Recall that $\haT_0=\frac{1}{\sqq-\sqq^{-1}}$ and $\haT_{-1}=0$.
		A direct computation using Proposition~ \ref{prop:ThetaO} and Proposition~ \ref{prop:OTheta} shows that $\text{LHS } \eqref{TOTO} - \text{RHS } \eqref{TOTO} =0$.
	\end{proof}

	\section{Derived equivalence and two presentations} 
	\label{sec:derived}
	
	Let $\LaK^\imath$ be the $\imath$quiver algebra of the Kronecker $\imath$quiver $(\QK,\Id)$; see \eqref{eq:iquiver}.
	In this section, we establish a derived equivalence between categories $\rep(\LaK^\imath)$ and $\calc_1(\coh(\P^1_\bfk))$, inducing an isomorphism of the $q$-Onsager algebra in the Serre type and the Drinfeld type presentations.
	Consequently, the homomorphism $\Omega$ from the $q$-Onsager algebra to the $\imath$Hall algebra of the projective line is injective.

	\subsection{Reflection functors}
	
	Let $(Q, \varrho)$ be a (connected) acyclic $\imath$quiver of rank $\ge 2$.
	We recall the reflection functors for the $\imath$quiver algebra $\Lambda^\imath$ associated to a split $\imath$quiver $(Q,\Id)$ from \cite{LW21a}.
	For any sink $\ell \in Q_0$, define the quiver $s_\ell (Q)$ by reversing all the arrows ending at $\ell $.
	Let $\Lambda'^\imath =\bs^+_\ell \Lambda^{\imath}$ be the $\imath$quiver algebra of the split $\imath$quiver $(Q',\Id)$.
	The quiver $\ov{Q'}$ of $\bs^+_\ell \Lambda^{\imath}$ can be constructed from $\ov{Q}$ by reversing all the arrows of $Q$ ending at $\ell $.
	Dually, we can define $\bs^{-}_\ell\Lambda^\imath$ for any source $\ell$ by reversing all the arrows of $Q$ starting from $\ell $. 
	
	A {\rm reflection functor} associated to a sink $\ell \in Q_0$ is defined in \cite[\S3.2]{LW21a}
	\begin{align}  \label{eq:Fl}
		F_\ell ^+: \rep(\Lambda^\imath) \stackrel{\text{def}}{=} \rep(\ov{Q},\ov{I}) \longrightarrow \rep(\bs^+_\ell \Lambda^\imath) \stackrel{\text{def}}{=} \rep(\ov{Q'},\ov{I'}).
	\end{align}
	Dually, a reflection functor $F_\ell^-$ associated to a source $\ell\in Q_0$ is also defined. In fact, the action of $F_\ell ^+$ (and $F_\ell^-$) on $\rep_\bfk( Q)\subseteq \rep(\Lambda^\imath)$ is the same as that of the classic BGP reflection functor.

	\subsection{The Kronecker $\imath$quiver}
	
	Consider the Kronecker quiver $\QK: \xymatrix{0\ar@<0.5ex>[r]^\alpha \ar@<-0.5ex>[r]_\beta& 1}$. Denote by $A^t$ the transpose of a matrix $A$.
	
	\begin{lemma}[see e.g. \text{\cite[Section 3.2]{Rin84}}]
		\begin{enumerate}
			\item
			Indecomposable pre-projective $\bfk \QK$-modules (up to isomorphisms) are given by:
			\[
			P_n=\xymatrix{\bfk^n\ar@<0.5ex>[rr]^{A_n^{\rm pro}} \ar@<-0.5ex>[rr]_{B_n^{\rm pro}}&& \bfk^{n+1}}, \quad \forall n\geq0,
			\]
			where
			\[
			A_n^{\rm pro} = \left( \begin{array}{cccc}
				1&0&\cdots &0 \\
				0&1&\cdots &0 \\
				\vdots&\vdots&\ddots &\vdots\\
				0&0&\cdots&1\\
				0&0&\cdots&0
			\end{array} \right) \text{ and } B_n^{\rm pro} = \left( \begin{array}{cccc}
				0&0&\cdots&0\\
				1&0&\cdots &0 \\
				0&1&\cdots &0 \\
				\vdots&\vdots&\ddots &\vdots\\
				0&0&\cdots&1
			\end{array} \right).
			\]
			
			\item
			Indecomposable pre-injective $\bfk \QK$-modules (up to isomorphisms) are given by:
			\[
			I_n=\xymatrix{\bfk^{n+1}\ar@<0.5ex>[rr]^{(A_n^{\rm pro})^t} \ar@<-0.5ex>[rr]_{(B_n^{\rm pro})^t}&& \bfk^{n}}, \quad \forall n\geq0.
			\]
		\end{enumerate}
		%
	\end{lemma}
	
	Let $\LaK^\imath$ be the $\imath$quiver algebra of the split Kronecker $\imath$quiver $(\QK,\Id)$, cf. \cite{LW20}.
	Then $\LaK^\imath$ is isomorphic to the algebra with its quiver $\ov \QK$ and relations as follows:
	\begin{center}\setlength{\unitlength}{0.7mm}
		\vspace{-0.2cm}
		\begin{equation}
			\label{eq:iquiver}
			\begin{picture}(30,13)(0,0)
				\put(0,-2){\small $0$}
				\put(2.5,0.5){\vector(1,0){17}}
				\put(2.5,-2){\vector(1,0){17}}
				\put(10,-0.5){$^{\alpha}$}
				\put(10,-8.5){$^{\beta}$}
				\put(20,-2){\small $1$}
				\color{purple}
				\put(0,9){\small $\varepsilon_0$}
				\put(20,9){\small $\varepsilon_1$}
				
				\qbezier(-1,1)(-3,3)(-2,5.5)
				\qbezier(-2,5.5)(1,9)(4,5.5)
				\qbezier(4,5.5)(5,3)(3,1)
				\put(3.1,1.4){\vector(-1,-1){0.3}}
				
				\qbezier(19,1)(17,3)(18,5.5)
				\qbezier(18,5.5)(21,9)(24,5.5)
				\qbezier(24,5.5)(25,3)(23,1)
				\put(23.1,1.4){\vector(-1,-1){0.3}}
			\end{picture}
		\end{equation}
		\vspace{-0.2cm}
	\end{center}
	\[
	\varepsilon_0^2=0=\varepsilon_1^2, \quad \varepsilon_1 \alpha=\alpha\varepsilon_0, \quad \varepsilon_1 \beta=\beta\varepsilon_0.
	\]
	We can make sense of the $\imath$Hall algebra $\tMHL$; see Example~\ref{ex:Q}. Recall the universal $q$-Onsager algebra $\tUi$ from \S\ref{subsec:On}.
	
	\begin{proposition}
		[\text{\cite[Theorem 9.6]{LW20}}]
		\label{prop:mono}
		There exists a $\Q({\sqq})$-algebra monomorphism
		\begin{align}
			\label{eq:psi}
			\widetilde{\psi}: \tUi_{ |v={\sqq}} &\longrightarrow \tMHL,
		\end{align}
		which sends
		$
		B_i \mapsto \frac{-1}{q-1}[S_{i}],
		\K_i \mapsto [K_{S_i}], \text{ for }i=0,1.
		$
	\end{proposition}
	
	Let $K_0(\bfk \QK)=\Z^2$ be the Grothendieck group of $\rep_\bfk(\QK)$. Denote by $\alpha_i$ the class of $S_i$, for $i=0,1$. For $\beta=a_0\alpha_0+ a_1\alpha_1\in \Z^2$, we can define $[K_\beta]:=[K_{S_0}]^{a_0}*[K_{S_1}]^{a_1}$. Let $\widetilde{\ct}(\LaK^\imath)$ be the quantum torus of $\tMHL$, i.e., the subalgebra of $\tMHL$ generated by $[K_{S_i}]$, $i=0,1$. Then $\tMHL$ is free as a left $\widetilde{\ct}(\LaK^\imath)$-module with a basis given by
	$\{[X]\mid  X\in\rep_\bfk(\QK)\subseteq \rep(\LaK^\imath)\}$; see Proposition \ref{prop:hallbasis}. 
	
	Let $\QK':=\bs^+_1 \QK=\bs^-_0 \QK$, and let $\LaK'^\imath$ be the split $\imath$quiver algebra of $(\QK', \Id)$. Let $\dag: \QK'\rightarrow \QK$ be the natural isomorphism by exchanging the vertices $0$ and $1$, which induces an isomorphism $\dag:\rep(\LaK'^\imath)\stackrel{\simeq}{\rightarrow}\rep(\LaK^\imath)$.
	Then we obtain two functors (cf. \eqref{eq:Fl} for $F_i^\pm$):
	\begin{align*}
		\BS^-:=\dag \circ F_0^-: \rep(\LaK^\imath)\rightarrow \rep(\LaK^\imath),
		\quad
		\BS^+:=\dag \circ F_1^+: \rep(\LaK^\imath)\rightarrow \rep(\LaK^\imath).
	\end{align*}
	We shall describe the actions of $\BS^+$ and $\BS^-$ on $\rep_\bfk(\QK)\subseteq \rep(\LaK^\imath)$.
	
	For a representation of $\bfk \QK$,
	$X= \xymatrix{(V\ar@<0.5ex>[r]^A \ar@<-0.5ex>[r]_B& W)},$
	we have the  exact sequences
	\begin{align*}
		0\rightarrow U'\stackrel{\begin{pmatrix} C'\\D'\end{pmatrix}}{\longrightarrow} V\oplus V\stackrel{(A,B)}{\longrightarrow} W ,
		\quad
		V \stackrel{\begin{pmatrix} A\\B\end{pmatrix}}{\longrightarrow} W\oplus W\stackrel{(C'',D'')}{\longrightarrow} U''\rightarrow0.
	\end{align*}
	Then by \cite[Section 1]{BGP73} we have
	\begin{align*}
		\BS^+(X)= \xymatrix{(U'\ar@<0.5ex>[r]^{C'} \ar@<-0.5ex>[r]_{D'}& V)}, \quad
		\BS^-(X) = \xymatrix{(W\ar@<0.5ex>[r]^{C''} \ar@<-0.5ex>[r]_{D''}& U'')}.
	\end{align*}
	
	Let $\rep^i_\bfk(\QK)$ be the full subcategory of $\rep_\bfk(\QK)$ consisting of modules without summands isomorphic to $S_i$, for $i=0,1$. Then $\BS^+:\rep^1_\bfk(\QK) \rightarrow\rep^0_\bfk(\QK)$ is an equivalence with $\BS^-$ as its inverse.
	
	The functor $F_1^+$  (cf. \eqref{eq:Fl}) induces an isomorphism $\Gamma_1: \tMHL  \stackrel{\sim}{\rightarrow}  {}^\imath\widetilde{\ch}(\bfk\QK')$ by \cite[Theorem 4.3]{LW21a}.
	Similarly, $\BS^+$ induces an automorphism $\BS^+: \tMHL  \stackrel{\sim}{\rightarrow}  {}^\imath\widetilde{\ch}(\bfk\QK')$.
	Note that $\dag$ induces an isomorphism $\dag:   {}^\imath\widetilde{\ch}(\bfk\QK')   \stackrel{\sim}{\rightarrow}  \tMHL$. By definition, it is clear that
	\[
	\BS^+=\dag \circ \Gamma_1 :  \tMHL  \stackrel{\sim}{\longrightarrow} \tMHL.
	\]
	By \cite[Proposition 4.4]{LW21a}, we have the following result.
	
	\begin{lemma}
		\label{prop:reflection}
		The isomorphism $\BS^+:\tMHL\xrightarrow{\sim}  \tMHL$ sends, for $M\in\rep^1_\bfk(\QK)$,
		\begin{align*}
			\BS^+([M])&= [\BS^+(M)],
			\quad
			\BS^+([S_1]) = [K_{S_0}]^{-1}* [S_0],
			\\
			\BS^+([K_{\alpha_1}])&= [K_{ -\alpha_0}], \quad \BS^+([K_{\alpha_0}])= [K_{ \alpha_1+2\alpha_0}].
		\end{align*}
	\end{lemma}
	
	Dually, we have the following lemma.
	
	\begin{lemma}
		\label{prop:refinverse}
		The isomorphism $\BS^-:\tMHL\xrightarrow{\sim}  \tMHL$ sends, for $M\in\rep^0_\bfk(\QK),$
		\begin{align*}
			\BS^-([M])= [\BS^-(M)],  &
			\quad
			\BS^-([S_0]) = [K_{S_1}]^{-1}* [S_1],
			\\
			\BS^-([K_{\alpha_1}]) = [K_{ \alpha_0+2\alpha_1}], &
			\quad
			\BS^-([K_{\alpha_0}])= [K_{- \alpha_1}].
		\end{align*}
	\end{lemma}

	By adapting \cite{LW21a}, we have the following commutative diagrams
	\[
	\xymatrix{ \tUi_{ |v={\sqq}} \ar[r]^{\dag T_1} \ar[d]^{\widetilde{\psi}} & \tUi_{ |v={\sqq}} \ar[d]^{\widetilde{\psi}}
		\\
		\tMHL \ar[r]^{\BS^+}  &\tMHL}
	\qquad\qquad
	\xymatrix{ \tUi_{ |v={\sqq}} \ar[r]^{T_1^{-1} \dag} \ar[d]^{\widetilde{\psi}} & \tUi_{ |v={\sqq}} \ar[d]^{\widetilde{\psi}}
		\\
		\tMHL \ar[r]^{\BS^-}  &\tMHL}
	\]
	
	\begin{lemma}
		\label{lem:proj-inj}
		For $n\geq0$, we have
		\begin{align*}
			\widetilde{\psi}(B_{n\delta+\alpha_1})=\frac{-1}{q-1}[P_n],  &\qquad
			\widetilde{\psi}(B_{-(n+1)\delta+\alpha_1})   =\frac{-1}{q-1}[I_n] *[K_{-n\de-\alpha_0}].
		\end{align*}
	\end{lemma}
	
	\begin{proof}
		We have $(\BS^-)^n(M)\in \rep^0_\bfk(\QK)$, and for any indecomposable $M\in \rep^0_\bfk(\QK)$, we have $\BS^-(M)$ is the unique $\bfk \QK$-module (up to isomorphisms) of dimension  $\dag \bs_0(\dimv M)$. In particular, $\dimv(\BS^-)^n(S_1) =n\de+\alpha_1$, so $(\BS^-)^n(S_1)\cong P_n$.
		Then
		\begin{align*}
			\widetilde{\psi}(B_{n\delta+\alpha_1})=& \widetilde{\psi}( ( T_0\dag)^{-1}(B_1))=(\BS^-)^n \widetilde{\psi}(B_1)
			\\
			=&\frac{-1}{q-1}(\BS^-)^n([S_1])
			=\frac{-1}{q-1}[P_n].
		\end{align*}
		We also have
		\begin{align*}
			\widetilde{\psi}(B_{-(n+1)\delta+\alpha_1})=&\widetilde{\psi}((T_0\dag)^{n+1}(B_1))=(\BS^+)^{n+1}\widetilde{\psi}(B_1)
			\\
			=&\frac{-1}{q-1}(\BS^+)^{n+1}([S_1])=\frac{-1}{q-1}(\BS^+)^{n}([S_0]*[K_{-\alpha_0}])
			\\
			=&\frac{-1}{q-1}(\BS^+)^{n}([S_0])*[K_{-n\de-\alpha_0}]=\frac{-1}{q-1}[I_n] *[K_{-n\de-\alpha_0}].
		\end{align*}
		The lemma is proved.
	\end{proof}

	\subsection{A derived equivalence}
	
	Let $\tCMH$  be the composition subalgebra of $\tMHL$ generated by $[S_i]$, and $[K_{S_i}]$ ($i=0,1$).
	
	\begin{lemma}
		The composition algebra $\tCMH$ contains all the elements $[X]$, where $X$ is either an indecomposable pre-injective module or an indecomposable pre-projective module.
	\end{lemma}
	
	\begin{proof}
		By Proposition~\ref{prop:mono}, we have an algebra isomorphism
		$\widetilde{\psi}: \tUi_{|v={\sqq}} \stackrel{\cong}{\rightarrow} \tCMH$ given by \eqref{eq:psi}.
		Now the lemma follows from Lemma \ref{lem:proj-inj}.
	\end{proof}
	
	\begin{definition}
		The composition algebra $\tCMHP$ is the subalgebra of $\tMHX$ generated by the elements $[\co(n)]$, $\haT_k$, and $[K_\alpha]$ where $n\in\Z,k\geq1$ and $\alpha\in K_0(\PL)\cong \Z^2$.
	\end{definition}
	
	Let $T=\co\oplus\co(1)$ and $B=\End_{\PL}(T)$. It is known \cite{Bei79} that $T$ is a tilting object, and $B^{op}\cong \bfk \QK$. It follows that
	\begin{align}  \label{Bei}
		\RHom_{\PL}(T,-):\cd^b(\coh(\PL))\stackrel{\simeq}{\longrightarrow} \cd^b(\rep_\bfk( \QK))
	\end{align}
	is a derived equivalence.
	
	Let $\cv$ be the subcategory of $\coh(\PL)$ consisting of $M$ such that $\Hom_{\PL}(T,M)=0$. Denote $\cu=\Fac T$, the full subcategory of $\coh(\PL)$ consisting of homomorphic images of objects in $\add T$. Then $(\cu,\cv)$ is a {\em torsion pair} of $\coh(\PL)$.
	
	\begin{lemma}
		$(\calc_1(\cu),\calc_1(\cv))$ is a torsion pair of $\calc_1(\coh(\P_\bfk^1))$. In particular, any $M\in\calc_1(\coh(\P_\bfk^1))$ admits a short exact sequence of the form
		\begin{align}
			\label{T-resol}
			0 \longrightarrow M \longrightarrow X_M \longrightarrow T_M \longrightarrow 0
		\end{align}
		where $X_M\in \calc_1(\cu)$ and $T_M\in\add K_T$.
	\end{lemma}
	
	\begin{proof}
		We have $\Hom(U^\bullet,V^\bullet)=0$ for any $U^\bullet\in\calc_1(\cu)$, $V^\bullet\in\calc_1(\cv)$.
		
		For any $M^\bullet=(M,d)\in\calc_1(\coh(\P_\bfk^1))$, there exists
		a short exact sequence $0\rightarrow U\xrightarrow{f} M\xrightarrow{g} V\rightarrow0$ such that $U\in\cu$, $V\in\cv$.
		Since $\Hom(U,V)=0$, we have $gdf=0$, and then there exists $d':U\rightarrow U$ and $d'':V\rightarrow V$ such that $df=fd'$ and $gd=d''g$.
		Denote by $U^\bullet=(U,d')$ and $V^\bullet=(V,d'')$. We  have the  exact sequence
		$0\rightarrow U^\bullet\rightarrow M^\bullet\rightarrow V^\bullet\rightarrow0.$
		So $(\calc_1(\cu),\calc_1(\cv))$ is a torsion pair of $\calc_1(\coh(\P_\bfk^1))$.
		
		The proof of the last statement is the same as for \cite[Theorem 5.8]{LP21}.
	\end{proof}

	\begin{proposition}  \label{prop:tilt}
		Let $K_T=K_{\co}\oplus K_{\co(1)}$. Then $K_T$ is a tilting object of $\cd^b(\calc_1(\coh(\PL)))$ with
		$\End(K_T)\cong \LaK^\imath$, which gives rise to an equivalence:
		$$\RHom (K_T,-): \cd^b(\calc_1(\coh(\PL))) \stackrel{\simeq}{\longrightarrow} \cd^b(\rep(\LaK^\imath)).$$
	\end{proposition}
	
	\begin{proof}
		Similar to \cite[Theorem 5.11]{LP21}, one has that $\Ext^p(K_T,K_T)=0$ for $p>0$.
		We have $\calc_1(\cu)=\Fac(K_T)$, and hence \eqref{T-resol} implies that $K_T$ is a tilting object of $\cd^b(\calc_1(\coh(\PL)))$.
		It is routine to check that $\End(K_T)\cong \LaK^\imath$, which will be omitted. Then the derived equivalence follows from the standard arguments; see \cite{Ha88}.
	\end{proof}
	
	\begin{proposition}
		\label{prop:F}
		Let $F= \Hom(K_T,-)$. Then there exists an isomorphism
		\begin{align*}
			\BF:\tMHX&\stackrel{\sim}{\longrightarrow} \tMHL\\\notag
			[M]&\mapsto [F(T_M)]^{-1}* [F(X_M)],
		\end{align*}
		where $X_M\in \calc_1(\cu)$ and $T_M\in\add K_T$, are defined in the short exact sequence \eqref{T-resol}.
	\end{proposition}
	
	\begin{proof}
		Follows from \cite[Theorem A. 22]{LW22} with the help of Lemma~\ref{prop:tilt}.
	\end{proof}

	\subsection{Injectivity of $\Omega$}
	
	Recall the isomorphism $\Phi: \tUiD \rightarrow \tUi$ from Theorem~\ref{thm:UUiso}, the homomorphism $\Omega: \tUiD_{ |v=\sqq} \rightarrow \tMHX$ from Theorem~\ref{thm:morphi}, the monomorphism $\widetilde{\psi}: \tUi_{|v=\sqq} \rightarrow \tMHL$ from Proposition~\ref{prop:mono}, and the isomorphism $\BF:\tMHX \rightarrow \tMHL$ from Proposition~\ref{prop:F}.
	
	\begin{theorem}
		\label{main thm2}
		We have the following commutative diagram of algebra homomorphisms
		\[
		\xymatrix{\tUi_{ |v={\sqq}} \ar[d]^{\widetilde{\psi}} \ar[r]^{\Phi^{-1}}   & {}^{\text{Dr}}\tUi_{ |v=\sqq} \ar[d]^{\Omega}
			\\
			\tMHL \ar[r]^{\BF^{-1}} & \tMHX,  }\]
		where $\Phi, \,\BF$ are isomorphisms.
		In particular, the homomorphism $\Omega$ is injective.
	\end{theorem}
	
	\begin{proof}
		In this proof, we denote $\BG =\BF^{-1}: \tMHL\stackrel{\sim}{\rightarrow} \tMHX$ for short.
		Note that $\BG([S_1])=[\co]$ and $\BG([P_0])=[\co(1)]$. Then
		$
		\Omega\circ\Phi^{-1}(B_1) =\Omega (B_0)=\frac{-1}{q-1}[\co],
		$ 
		and
		\begin{align*}
			\BG\circ\widetilde{\psi}(B_1)&=\frac{-1}{q-1}\BG ([S_1])=\frac{-1}{q-1}[\co]=\Omega\circ\Phi^{-1}(B_1).
		\end{align*}
		
		We have a short exact sequence $0\rightarrow S_1^{\oplus 2}\stackrel{f}{\rightarrow} P_0\stackrel{g}{\rightarrow} S_0\rightarrow0$, which gives rise to a short exact sequence
		\[
		0\longrightarrow K_{S_1^{\oplus2 }} \longrightarrow M^\bullet \longrightarrow S_0\longrightarrow0,
		\]
		where
		\[
		M^\bullet=\Big( P_0\oplus S_1^{\oplus 2}, \left(\begin{array}{ccc} 0& f \\0&0  \end{array}\right)\Big).
		\]
		Hence we have $[S_0]=[M^\bullet]*[K_{S_1^{\oplus2}}]^{-1}$.
		
		Since $\Hom_{\PL}(\co,\co(1))\cong \Hom_{\LaK^\imath}(S_1,P_0)$, there exists a unique $h: \co^{\oplus 2}\rightarrow \co(1)$ such that $\Hom(T,h)=f$.
		Then
		$\Hom(K_T,X^\bullet)=M^\bullet$, where
		\[
		X^\bullet= \Big( \co(1)\oplus \co^{\oplus 2} , \left(\begin{array}{ccc} 0& h \\0&0  \end{array}\right)\Big).
		\]
		We have $\BF([X^\bullet])=[M^\bullet]$ thanks to $X^\bullet\in\calc_1(\cu)$.
		Note that $h$ is surjective, and $\ker h\cong \co(-1)$. Then
		$[X^\bullet]=[\co(-1)]*[K_{\co(1)}]= [\co(-1)]*[K_{\de+\widehat{\co}}]$.
		Hence we have
		\begin{align*}
			\BG([S_0])=\BG([M^\bullet]*[K_{S_1^{\oplus2}}]^{-1})=[\co(-1)]*[K_{\de+\widehat{\co}}]*[K_{-2\widehat{\co}}]=[\co(-1)]*[K_{\de-\widehat{\co}}].
		\end{align*}
		Therefore, we have
		\begin{align*}
			\Omega\circ\Phi^{-1}(B_0)&=\Omega (\y_{1,-1}C \K_1^{-1})=\frac{-1}{q-1}[\co(-1)]*[K_{\de-\widehat{\co}}]
			\\
			&=\frac{-1}{q-1}\BG([S_0])=\BG\circ\widetilde{\psi}(B_0).
		\end{align*}
		
		Finally, we verify that
		\[
		\Omega\circ\Phi^{-1}(\K_1)=\Omega(\K_1)= [K_\co]=  \BG([K_{S_1}])=\BG\circ \widetilde{\psi}(\K_1),
		\]
		and
		\begin{align*}
			\Omega\circ\Phi^{-1}(\K_{0})=&\Omega(C \K_1^{-1})= [K_\de]*[K_\co]^{-1}
			\\
			=& [K_{\co(1)}]*[K_{\co}]^{-2}
			= \BG([K_{P_0}]*[K_{S_1}]^{-2})
			\\
			=&\BG([K_{S_0}]) =\BG\circ\widetilde{\psi}(\K_0).
		\end{align*}
		Summarizing, we have proved $\Omega\circ\Phi^{-1}= \BG\circ\widetilde{\psi}$.
		
		The injectivity of $\Omega$ follows by the injectivity of $\widetilde{\psi}$ and the commutative diagram.
	\end{proof}
	
	\begin{corollary}
		The algebra isomorphism $\BF: \tMHX\stackrel{\sim}{\rightarrow} \tMHL$ restricts to an isomorphism of the composition  algebras
		$\BF: \tCMHP\stackrel{\simeq}{\rightarrow} \tCMH$.
		We have the following commutative diagram of isomorphisms:
		\[
		\xymatrix{ \tUiD_{|_{v=\sqq}} \ar[d]^{\Omega} \ar[r]^{\Phi}   & \tUi_{|_{v=\sqq}} \ar[d]^{\widetilde{\psi}}
			\\
			\tCMHP \ar[r]^{\BF} & \tCMH.  }\]
	\end{corollary}

	\section{A Hall algebra realization of imaginary root vectors}
	\label{sec:T}

	In this section, we provide a Hall algebra interpretation $\haH_m$ of the generators $H_m$ in the Drinfeld type presentation of the $q$-Onsager algebra.
	
	%
	
	Recall $\haT_m \in \tMHX$ from \eqref{def:Theta} is the image of $\Theta_m$ under the monomorphism $\Omega: \tUiD_{|_{v=\sqq}} \rightarrow \tMHX$; cf. Theorem~\ref{thm:morphi}. In light of \eqref{eq:exp1}, the elements
	\[
	\haH_m :=\Omega(H_m),
	\qquad \text{ for } m\geq1,
	\]
	must satisfy
	\begin{align}
		\label{eq:Tr}
		1+ \sum_{m\geq 1} (\sqq-\sqq^{-1})\haT_{m} z^m  = \exp\big( (\sqq-\sqq^{-1}) \sum_{m\ge 1} \haH_m z^m \big).
	\end{align}
	
	
	
	
	\begin{lemma}
		\label{lem:TO2}
		For $m\ge 1$ and $r\in \Z$, we have
		\begin{align}  \label{eq:HO3}
			[\haH_{m},[\co(r)]]= \frac{[2m]}{m} [\co(m+r)]- \frac{[2m]}{m} [\co(m-r)]* [K_{m\delta}].
		\end{align}
	\end{lemma}
	
	\begin{proof}
		Note that the equivalence between \eqref{iDR2} and \eqref{eq:hB1} follows from the identity \eqref{eq:exp1} formally, as shown in \cite{LW21b}. In the same way, the identity \eqref{eq:HO3} is equivalent to \eqref{TOTO}. 
	\end{proof}
	
	
	We shall describe the elements $\haH_m$.
	Recall that any indecomposable object in $\tor_x(\PL)$, for $x\in\PL$, has the form $S_x^{(n)}$ of length $n\geq 1$. For any partition $\lambda=(\lambda_1,\dots,\lambda_r)$, define
	\[S_x^{(\lambda)}:=S_x^{(\lambda_1)}\oplus\cdots \oplus S_x^{(\lambda_r)}. \]
	
	For any $x\in\PL$ and $m\geq 1$, we define (compare \cite{LRW21})
	\begin{align}  \label{eq:PTh}
		\haP_{m,x}:=\sum\limits_{\lambda\vdash m} n_x(\ell(\lambda)-1)\frac{[S_x^{(\lambda)}]}{|\Aut(S_x^{(\lambda)})|},
		\qquad
		\haT_{m,x}:=\frac{[S_x^{(m)}]}{\sqq_x-\sqq_x^{-1}},
	\end{align}
	where
	\[
	n_x(l)=\prod_{i=1}^l(1-\sqq_x^{2i})=\prod_{i=1}^l(1-q_x^{i}).
	\]
	Introduce the generating functions
	\begin{align}
		\haP_x(z):=&\sum_{m\geq1} \haP_{m,x} z^{m-1},
		\label{eq:Px} \\
		\haT_x(z):=&1+\sum_{m\geq1} (\sqq_x-\sqq_x^{-1})\haT_{m,x} z^{m}
		=1+\sum_{m\geq1} [S_x^{(m)}]z^{m}.
		\label{eq:Thx}
	\end{align}
	
	\begin{lemma}[\cite{LRW21}]
		We have
		\begin{align}\label{exp pression between theta and Hm}
			\haT_x(z)  =\exp\big( (\sqq_x-\sqq_x^{-1}) \sum_{m\ge 1} \haH_{m,x} z^{m} \big),
		\end{align}
		where
		\begin{align}  \label{eq:Tx}
			\haH_{m,x} =\sqq_x^{m} \frac{[m]_{\sqq_x}}{m} \haP_{m,x} -\delta_{m, ev} \sqq_x^{\frac{m}{2}} \frac{[m/2]_{\sqq_x}}{m} [K_{\frac{m}{2}d_x\de}].
		\end{align}
	\end{lemma}
	
	Recall $\haT_m$ from \eqref{def:Theta}.
	Define the generating function
	\begin{align}  \label{eq:haT}
		\haT(z) =1 + \sum_{m\ge 1} (\sqq -\sqq^{-1}) \haT_m z^m.
	\end{align}
	Since the categories $\tor_x(\PL)$ for $x \in \PL$ are orthogonal, by \eqref{def:Theta} and \eqref{eq:Thx} we have
	\begin{align*}
		\haT(\sqq z) =1+\sum_{m\geq 1}\sum_{||\bn||=m}[S_{\bn}]z^m=\prod_{x\in \PL} \haT_x(z^{d_x}).\\
	\end{align*}

	

	
	Now we establish the main result of this section.
	
	\begin{proposition}
		\label{prop:HaH}
		For $m\geq1$, we have
		\begin{align}\label{formula for Hm}
			\haH_m
			=&\sum_{x,d_x|m} \frac{[m]}{m} d_x \sum_{|\lambda|=\frac{m}{{d_x}}} n_x(\ell(\lambda)-1)\frac{[S_x^{(\lambda)}]}{\big|\aut(S_x^{(\lambda)})\big|}  -\de_{m,ev}\frac{[m]}{m} [K_{\frac{m}{2}\de}].
		\end{align}
	\end{proposition}
	
	\begin{proof}
		By \eqref{exp pression between theta and Hm},
		we have
		\begin{align*}
			\haT(\sqq z)=&\prod_{x\in \PL} \haT_x(z^{d_x})\\
			=&\prod_{x\in \PL} \exp\big( (\sqq_x-\sqq_x^{-1}) \sum_{m\ge 1} \haH_{m,x} z^{md_x} \big)\\
			=&\exp\big(\sum_{x\in \PL} (\sqq_x-\sqq_x^{-1}) \sum_{m\ge 1} \haH_{m,x} z^{md_x} \big)\\
			=&\exp\big(\sum_{m\geq 1}  \sum_{x, d_x|m} (\sqq_x-\sqq_x^{-1})\haH_{\frac{m}{d_x},x} z^{m} \big)\\
			=&\exp\big(  \sum_{m\ge 1}\sum_{x, d_x|m}(\sqq_x-\sqq_x^{-1}) \big(\sqq^{m} \frac{[m/d_x]_{\sqq_x}}{m/d_x} \haP_{\frac{m}{d_x},x}-\delta_{\frac{m}{d_x}, ev} \sqq^{\frac{m}{2}} \frac{[m/(2d_x)]_{\sqq_x}}{m/d_x} [K_{\frac{m}{2}\de}]\big) z^m\big).
		\end{align*}
		Observe from \eqref{eq:Tr} that
		\begin{align*}
			\exp\big( (\sqq-\sqq^{-1}) \sum_{m\ge 1} \haH_m z^m \big)
			=&1+ \sum_{m\geq 1} (\sqq-\sqq^{-1})\haT_{m} z^m = \haT(z).
		\end{align*}
		So it suffices to show that
		\begin{align}
			\label{eq:derived}
			&{\sqq^{-m}}\sum_{x, d_x|m} (\sqq_x-\sqq_x^{-1})\big(\sqq^{m} \frac{[m/d_x]_{\sqq_x}}{m/d_x} \haP_{\frac{m}{d_x},x} -\delta_{\frac{m}{d_x}, ev} \sqq^{\frac{m}{2}} \frac{[m/(2d_x)]_{\sqq_x}}{m/d_x} [K_{\frac{m}{2}\de}]\big)\\\notag
			&={(\sqq-\sqq^{-1})}\Big(\sum_{x,d_x|m} \frac{[m]}{m} d_x \sum_{|\lambda|=\frac{m}{{d_x}}} n_x(\ell(\lambda)-1)\frac{[S_x^{(\lambda)}]}{\big|\aut(S_x^{(\lambda)})\big|}  -\de_{m,ev}\frac{[m]}{m} [K_{\frac{m}{2}\de}]\Big).
		\end{align}
		
		Recalling $\haP_{m,x}$ from \eqref{eq:PTh}, we have
		\begin{align*}
			&\sqq^{-m}\sum_{x, d_x|m}(\sqq_x-\sqq_x^{-1}) \sqq^{m} \frac{[m/d_x]_{\sqq_x}}{m/d_x} \haP_{\frac{m}{d_x},x} \\
			&=\sum_{x, d_x|m}  \frac{\sqq^m-\sqq^{-m}}{m} d_x \sum\limits_{\lambda\vdash \frac{m}{d_x}} n_x(\ell(\lambda)-1)\frac{[S_x^{(\lambda)}]}{|\Aut(S_x^{(\lambda)})|} \\
			&=(\sqq-\sqq^{-1})\sum_{x,d_x|m} \frac{[m]}{m} d_x \sum_{|\lambda|=\frac{m}{{d_x}}} n_x(\ell(\lambda)-1)\frac{[S_x^{(\lambda)}]}{\big|\aut(S_x^{(\lambda)})\big|}.
		\end{align*}
		Moreover, using \eqref{eq:FF}, we compute
		\begin{align*}
			&\sqq^{-m}\sum_{x, d_x|m}  (\sqq_x-\sqq_x^{-1})\delta_{\frac{m}{d_x}, ev}\sqq^{\frac{m}{2}} \frac{[m/(2d_x)]_{\sqq_x}}{m/d_x} [K_{\frac{m}{2}\de}]\\
			&=\sum_{x, d_x|m} \delta_{\frac{m}{d_x}, ev} \sqq^{-\frac{m}{2}} \frac{\sqq^{\frac{m}{2}}-\sqq^{-\frac{m}{2}}}{m} d_x [K_{\frac{m}{2}\de}]
			= \delta_{m, ev} \frac{1-\sqq^{-m}}{m} \sum_{x, d_x|\frac{m}{2}}d_x [K_{\frac{m}{2}\de}]\\
			&= \delta_{m, ev} \frac{1-\sqq^{-m}}{m} (1+q^{\frac{m}{2}}) [K_{\frac{m}{2}\de}]
			=(\sqq-\sqq^{-1})\de_{m,ev}\frac{[m]}{m} [K_{\frac{m}{2}\de}].
		\end{align*}
		So \eqref{eq:derived} holds. We are done.
	\end{proof}


\end{document}